\documentclass[Threeside,11pt,tikz,border=3mm]{article}
\usepackage{etoolbox}
\makeatletter
\patchcmd{\thebibliography}
{\list}
{\small          
	\setstretch{1.2}
	\list}
{}{}
\makeatother
\usepackage{setspace} 
\usepackage{tikz}
\usepackage{tikz-cd}
\usetikzlibrary{arrows.meta}
\usepackage{subcaption}
\usepackage [latin1]{inputenc}
\usepackage[english]{babel}
\usepackage{amsmath,amscd}
\usepackage{amsthm}
\usepackage{algorithm}
\usepackage{amsfonts}
\usepackage{amssymb}
\usepackage{mathrsfs}
\usepackage{graphicx}
\usepackage{colortbl,dcolumn}
\usepackage{marvosym}
\usepackage{ifsym}
\usepackage{paralist}
\usepackage{xcolor}
\usepackage{array,color}
\usepackage{geometry}
\usepackage[colorlinks,
citecolor=blue,
linkcolor=blue,
backref=page
allcolors=black]{hyperref} 
\usetikzlibrary{shapes.geometric,arrows.meta,calc}
\usetikzlibrary{arrows, decorations.markings}
\usetikzlibrary{patterns}
\graphicspath{{figures/}}
\allowdisplaybreaks
\newtheorem{lemma}{\bf Lemma}[section]
\newtheorem{theorem}{\bf Theorem}[section]

\newtheorem{proposition}{\bf Proposition}[section]
\newtheorem{corollary}{\bf Corollary}[section]
\newtheorem{definition}{\bf Definition}[section]
\newtheorem{remark}{\bf Remark}[section]

\numberwithin{equation}{section}

\definecolor{mydeepgreen}{RGB}{3,100,50}
\usepackage{enumitem}
\begin{document}
\title{{\sl Relation between irrationality and regularity for $ C^1 $ conjugacy of $ C^2 $ circle diffeomorphisms to rigid rotations}}
\author{Zhicheng Tong $^{a,1}$, Shuyuan Xiao $^{b,2}$, Yong Li $^{*,c,1,3}$ }

\renewcommand{\thefootnote}{}
\footnotetext{\hspace*{-6mm}

\begin{tabular}{l l}
 $^{*}$~~~The corresponding author.\\
 $^{a}$~~~E-mail address : tongzc25@jlu.edu.cn\\
 $^{b}$~~~E-mail address : xiaosy@dlnu.edu.cn\\
 $^{c}$~~~E-mail address : liyong@jlu.edu.cn\\
 $^{1}$~~~School of Mathematics, Jilin University, Changchun 130012,  China.\\
 $^{2}$~~~School of Mathematical Sciences, Dalian Minzu University, Dalian 116600,  China.\\
$^{3}$~~~Center for Mathematics and Interdisciplinary Sciences, Northeast Normal University, Changchun 130024,  China.
\end{tabular}}

\date{}
\maketitle

\begin{abstract}
By introducing the modulus of continuity, we first establish the corresponding cross-ratio distortion estimates under  $ C^2 $ smoothness, and further derive a Denjoy-type inequality, which is almost optimal for dealing with circle diffeomorphisms. The latter plays a prominent role in the study of $ C^1 $ conjugacy to irrational rotations. We also establish an explicit integrability correlation between continuity and irrationality for the first time. Furthermore, the regularity of the conjugation is addressed and proved to be sharp.
\end{abstract}

\tableofcontents

\section{Introduction}
Conjugacy on the circle, that is, determining under what conditions a diffeomorphism on $ \mathbb{T}^1=\mathbb{R} /  \mathbb{Z}=\left[ {0,1} \right) $ can be conjugated to a rotation, is one of the most fundamental yet difficult topics in dynamical systems. It has a long history of research. It has been known since Poincar\'{e} that the rotation number $ \rho(T) \in \mathbb{T}^1 $ is always well defined (up to an integer summand) for an orientation-preserving homeomorphism $ T $ on $ \mathbb{T}^1 $. More precisely,
\[\rho \left( T \right) = \mathop {\lim }\limits_{n \to  + \infty } \frac{{L_T^n{x_0}}}{n},\]
where $ {L_T}\left( x \right) $ is  the lifting of $ T $ from $ \mathbb{T}^1 $ onto $ \mathbb{R} $ such that $ {L_T}\left( {x + 1} \right) = {L_T}\left( x \right) + 1 $. It is worth mentioning that $ \rho(T) $ does not depend on the choice of the initial point $ x_0 \in \mathbb{T}^1 $, and it is irrational if and only if the mapping $ T $ has no periodic points. We will focus on this case throughout the  paper.  In the study of circle conjugacy, the smooth (even analytic) conjugacy is more important than the  topological type. It should be pointed out that not all irrational rotations can be smoothly conjugated, see the counterexamples constructed by Arnol'd \cite{Arnold} and Herman \cite{Herman79}. The latter showed that for a Liouville rotation number $ \alpha_1 $ (i.e. is not Diophantine and Bruno), there exists an analytical diffeomorphism $ \widetilde{T} $ of the circle for which the topological conjugacy to the rotation $ R_{\alpha_1}(\theta)=\theta+\alpha_1 \mod 1$ is not even absolutely continuous (certainly does not admit $ C^1 $ conjugacy), see also the work in \cite{MR2914151,Yoccoz84}.  Denjoy \cite{A.Denjoy} proved that if $ T $ is of class $ C^2 $ then it could topologically conjugate to the rotation $ R_{\alpha_2}(\theta)=\theta+\alpha_2 \mod 1$, i.e. there exists a homeomorphism $ \phi $ on $ \mathbb{T}^1 $ such that $ \phi \circ T= R_{\alpha_2} \circ \phi$.  Based on these crucial facts, a natural question to ask is:

\textit{Under what conditions is the conjugation smooth, or merely differentiable, or equivalently what is the relation between the regularity of $ T $ and the irrationality of the rotation number $ \rho(T) $ concerning at least $ C^1 $ conjugacy?}

Irrationality can be described in many ways.  To consider conjugacy on the circle, Herman \cite{Herman79} established a condition of partial quotients whose corresponding irrational numbers form a set of full Lebesgue measure. Incidentally, it is a breakthrough in global conjugacy, which is different  from the KAM approach relying on the localness (i.e., one has to require that $ T $ must be close to $ R_\rho $).  Yoccoz \cite{Yoccoz84} further adopted the nonresonant condition of Diophantine's type for the irrational rotation $ \rho \in \mathbb{T}^1 $ with exponent $ \beta>0 $, i.e.,
\begin{equation}\label{diobeta}
	\left| {\rho  - \frac{p}{q}} \right| > \frac{{C\left( \rho  \right)}}{{{q^{2 + \beta }}}}
\end{equation}
for any rational number $ p/q $, where $ C(\rho) $ is a positive number that  depends on $ \rho $.  In this case the regularity requirement of the mapping $ T $ is $ C^k $, where $ k\geqslant3 $ and $ k>2\beta+1 $, and the essentially unique diffeomorphism  which conjugates $ T $ to $ \rho $ is of class $ C^{k-1-\beta-\varepsilon} $ for every $ \varepsilon>0 $. Recently, a  result in the H\"older sense appeared in  \cite{Khanin09Invent} based on Diophantine rotation with exponent $ \delta $ (see \eqref{diobeta}), which weakens $ C^k $ smoothness for the mapping $ T $ to $ C^{2+\alpha} $, through an improved Denjoy's inequality of the H\"older type, where $ 0 \leqslant \delta  < \alpha  \leqslant 1 $ and $ \alpha  - \delta  < 1 $. The conclusion that the regularity of the conjugation at this point is of class $ {C^{1 + \alpha  - \delta }} $ is proved to be optimal, which sharpens the work  \cite{ETDScon} (the corresponding regularity is $ C^{1+\alpha  - \delta-\varepsilon} $ with any $ \varepsilon>0 $), i.e., the exponent $ 1 + \alpha  - \delta $ cannot be higher in general settings due to the counterexamples in \cite{ETDScon}. It should be emphasized that this first result on optimal regularity is based on the use of cross-ratio distortion estimates for $ T $  of class $ C^{2+\alpha} $ with $ \alpha>0 $, which is a conceptually new approach.

However, much recent work requires $ C^k $ (at least $ k \geqslant 2 $) plus certain  H\"older continuity for $ T $, and one of our motivations is to weaken the regularity  to only $ C^2 $ (and  preserves $ C^1 $ conjugacy simultaneously), which is the weakest case that might be achieved due to certain counterexamples, see \cite{MR662606,MR0581808,MR0482815}. This is quite necessary because H\"older continuity is not sufficient to deal with all problems of finite smoothness from the perspective of  Baire category. Therefore,  we are actually still a long way from the true critical situation. As an illustration, let us consider some $ T $ with continuous $ D^2 T= - {\left( {\log \left( {x\left( {1 - x} \right)} \right)} \right)^{ - 1}}$ (define $ D^2 T =0$ at end points $ 0 $ and $ 1 $) which cannot be characterized by arbitrary H\"older type, since $ {D^2}T \sim {\left( {\log {x^{ - 1}}} \right)^{ - 1}} $ as $ x \to 0^+ $, that is, the Logarithmic H\"older type with index $ 1 $ in our terminology (see Section \ref{SecModulus}), and thus this case cannot be analyzed through the well known results via only H\"older continuity.

In this paper, we extend cross-ratio distortion estimates to the weakest case and study $ C^1 $ conjugacy on the circle from a different perspective. This corresponds exactly to the opinion proposed by  Khanin and Teplinsky in \cite{Khanin09Invent}, who suggested that these powerful tools would prove useful in other problems (different from the H\"older regularity via Diophantine irrationality) involving circle diffeomorphisms.  Specifically, by introducing the definition of modulus of continuity in  Section \ref{SecModulus}, we first establish the cross-ratio distortion estimates for  any $ C^2 $ strictly monotone function, and then obtain the corresponding Denjoy-type inequality for any irrational rotation in Section \ref{Cross-ratio estimates via modulus of continuity}. Via these tools, we present the following main result (in Section \ref{Circle diffeomorphisms}), which contains  a sufficient integrability condition under which $ C^2 $  orientation-preserving circle   diffeomorphisms can be $ C^1 $ conjugated to irrational rotations (characterized by partial quotients rather Diophantine or Bruno  condition):
\\
\textit{\textbf{[Main Theorem]}\ Let $ T $ be a $ C_{2,\varpi} $  orientation-preserving circle diffeomorphism with	rotation number $ \rho(T)=[k_1,k_2, \ldots ] $, where $ k_{n+1}=\mathcal{O}(\varphi(n)) $, and $ \varpi $ is a modulus of continuity. Then the conjugation in Denjoy's theory is differentiable if
	\begin{equation}\notag
		\int_0^1 {\left( {\int_0^y {\varphi \left( {{{\log }_\lambda }x} \right)dx} } \right)\frac{{\varpi \left( y \right)}}{{{y^2}}}dy}  <  + \infty,
	\end{equation}
	where $ 0<\lambda<1 $ is a definite constant. Detailed definitions and notations are provided in Sections \ref{SecModulus} and \ref{Circle diffeomorphisms}.
}
\\
We then  further investigate the higher regularity of the conjugation, following the above main theorem, which slightly improves part of the optimal result in \cite{Khanin09Invent}. Besides, some explicit examples are provided, including the case of combining partial quotients with probability distributions.  Finally, in Section \ref{Optimality about our integrability condition} we show certain optimality of our results, from the perspective of preserving and improving $ C^1 $ conjugacy, etc.  To the best of our knowledge, our  integrability condition seems to explicitly link irrationality with regularity \textit{for the first time} and is relatively easy to verify. Now we obtain the optimal integrability condition about the irrationality and regularity.

Thereby \textit{our main theorem develops the classical Denjoy's and Herman's theory (see Sections \ref{Subdenjoy} and \ref{SubHerman} respectively) since the topological conjugacy is elevated to differentiable type and the smoothness assumption of $ T $ is reduced to only $ C^2 $ (in fact, $ C^2 $ plus some modulus of continuity which is much weaker than the H\"older type).} Actually, it should be emphasized that the $ C^2 $ regularity  for the mapping $ T $ cannot be further weakened, otherwise the conjugation might not be differentiable due to the loss of regularity and  irrationality of the rotation and might even become only pure singular, see Section \ref{Optimality about our integrability condition}.

\section{Modulus of continuity}\label{SecModulus}
To state the results in this paper, we first introduce the definition of modulus of continuity, which describes continuity weaker than the H\"older type. It has  attracted a lot of attention in dynamical systems, see for instance, Bufetov and Solomyak \cite{MR3209350}, Duarte and Klein \cite{MR3262622}, Fan and Jiang \cite{MR1860763}, Tong and Li \cite{CCM}, among others. Modulus of continuity  characterizes the complexity of dynamical systems, and obviously, not all dynamics can be preserved in dynamical systems of arbitrary complexity. Therefore, finding the inner connection between invariance and complexity and touching criticality is a fundamental but difficult topic. \textit{Generally, certain optimal integrability conditions for modulus of continuity always arise in critical cases, and they indeed solve some problems completely and perfectly.} Back to our concern,  Herman \cite{Herman79}  discussed the existence of circle diffeomorphisms via modulus of continuity and provided a stronger answer to Arnol'd's conjecture.  Katznelson and Ornstein \cite{ETDScon}  also considered problems similar to that in this paper (the terminology $ \mathscr{H}^{2+\psi} $ with a modulus of continuity $ \psi $ in Sections 2, 3 and 4 is equivalent to $ C_{2,\varpi}({\mathbb{T}^1}) $ here). See also Kim and Koberda \cite{MR4121156,MR4128339} on circle diffeomorphisms via the tool of modulus of continuity. 

To study the generalized Denjoy-type inequality provided in Section \ref{DenjoyDenjoy}, and further $ C^1 $ conjugacy for circle diffeomorphisms, even more accurate regularity in this paper  (including some optimal cases), we are in a position to formulate the definition of modulus of continuity, see Herman \cite{Herman79}. It should be noted that there are many definitions of modulus of continuity, and there are some subtle differences between them, for example, see  \cite{MR1860763,Herman79,CCM} and references therein.  We do not pursue that.

Throughout this paper, $ \mathcal{O}( \cdot ) $, $ o( \cdot ) $ and $  \sim  $ are uniform with respect to $ n , m \to +\infty$, or $ x \to 0^+ $ without causing ambiguity. Note that $ a\left( {n,m} \right) \sim b\left( {n,m} \right) $ implies that there exists a universal constant $ C_1,C_2>0 $ independent of $ n,m $ (may depend on other parameters), such that $ {C_1}b \leqslant a \leqslant {C_2}b $.

\begin{definition}\label{definition1}
	Denote by $ \varpi (x) $ a modulus of continuity, which is a strictly monotonic increasing continuous function defined on $ \mathbb{R}^+ $, such that 
	\begin{itemize}
		\item[(1)] $ \varpi \left( {x + y} \right) \leqslant \varpi \left( x \right) + \varpi \left( y \right) $ for $ x,y>0 $;
		
		\item[(2)] $ \varpi(px) \leqslant p\varpi(x) $ for $ p \in \mathbb{N}^+ $ and $ x>0 $;
		
		\item[(3)] $ \varpi(ax) \leqslant ([a]+1)\varpi(x) $ for $ a \in \mathbb{R}^+ $ and $ x>0 $.
	\end{itemize}
\end{definition}

The above definition is an extension of the classical Lipschitz and H\"older cases, which correspond to $ \varpi\sim x $ and $ \varpi \sim x^{\alpha} $ with some $ 0<\alpha<1 $, respectively. Further, there is a plethora of important examples that fail to be characterized by the $ \varepsilon $-H\"older type  for any $ \varepsilon \in (0,1) $. For instance, the Logarithmic H\"older type $ \varpi  \sim {\left( {\log {x^{ - 1}}} \right)^{ - \alpha }} $ with respect to $ 0 < x < 1 $ and index $ \alpha>0 $, and some examples generated from power series below. Note (1) implies that $ \varpi=\mathcal{O}(x) $, that is, could not be better than the Lipschitz type, otherwise all functions with this modulus of continuity must be some constants, see Definition \ref{con2.4}. In fact, these conditions hold automatically for most monotone functions of the type $ \varpi=\mathcal{O}(x) $, we do not pursue this point and focus on the order of $ \varpi $ at $ 0^+ $ instead, which is indeed essential to characterize the complexity of the continuity \textit{on a bounded domain throughout this paper}, that is, just focus on the case on the interval $ \left( {0,\delta } \right] $ with some $ \delta>0 $ instead of $ \mathbb{R}^+ $.

Although modulus of continuity can be defined in various ways, Dini integrability condition on modulus of continuity is universal and appears frequently in various fields, including dynamical systems and PDEs. Specifically, a modulus of continuity $ \varpi(x) $ on $ \left( {0,\delta } \right] $ with $ \delta>0 $ is said to satisfy the \textit{Dini condition} (or be of Dini continuous type) \cite{CCM} if
\begin{equation}\label{DINI}
	\int_0^\delta  {\frac{{\varpi \left( x \right)}}{x}dx}  <  + \infty.
\end{equation}
As an example, for the Logarithmic H\"older type $ \varpi  \sim {\left( {\log {x^{ - 1}}} \right)^{ - \alpha }} $ with $ \alpha>0 $ above, then $ \varpi $ satisfies the Dini condition \eqref{DINI} at this point if and only if $ \alpha>1 $. Note that the Dini condition can also be characterized with respect to the summability of evaluated along geometric sequences, see \cite{arxiv}. Namely, \eqref{DINI} holds if and only if
\[\sum\limits_{n = 1}^\infty  { \varpi \left( {{\theta ^n}\delta } \right)}  <  + \infty \]
for any $ \theta \in (0,1) $. Actually we can verify that
\begin{equation}\label{deltapiao}
	\widetilde \varpi \left( x \right): = \sum\limits_{n = 1}^\infty  { \varpi \left( {{\theta ^n}x} \right)} ,\quad x \in \left( {0,\delta } \right]
\end{equation}
is also a modulus of continuity at this point, see \cite{MR1860763}. Such series functions are obviously extremely difficult to analyze in specific problems, so it is necessary to introduce certain integrability conditions, such as  \eqref{DINI} and etc (in some cases the Dini condition might not be sufficient to preserve certain dynamical properties, such as regularity,  \textit{and thus stronger integrability conditions must arise as we will see later}).  Here we introduce a construction method about Dini continuous type, see \cite{arxiv}: Let $ {\left\{ {{a_j}} \right\}_{j \in {\mathbb{N}^ + }}} \in {c_0} $ and $ {\left\{ {{\gamma _j}} \right\}_{j \in {\mathbb{N}^ + }}} \in {\ell _1} $ be sequences of positive numbers, and $ \mathop {\lim }\nolimits_{j \to  + \infty } {a_j} = 0 $. We further assume that there exist $ \tau,{x_ * }>0  $  such that $  \tau  < {x_ * } $, and
\[\sum\limits_{j = 1}^\infty  {{a_j}x_ * ^{{\gamma _j}}}  <  + \infty ,\quad\sum\limits_{j = 1}^\infty  {{a_j}\frac{{{\tau ^{{\gamma _j}}}}}{{{\gamma _j}}}}  <  + \infty .\]
Then the function
\begin{equation}\label{powermoc}
	\varpi \left( x \right) \sim \sum\limits_{j = 1}^\infty  {{a_j}{x^{{\gamma _j}}}}
\end{equation}
is indeed a modulus of continuity on $ \left( {0,{x_ * }} \right] $, and satisfies the Dini condition \eqref{DINI}. As an explicit illustration, $ \varpi \left( x \right) = \sum\nolimits_{j = 1}^\infty  {\frac{{\sqrt[j]{x}}}{{{2^j}}}}  $.

Next we define the ``strong'' and ``weak'' properties, as well as the uniform continuity of the function  with respect to the modulus of continuity. Denote by $ \mathcal{D} $ a connected region in $ \mathbb{R}^d $ with some $ d \in \mathbb{N}^+ $, and $ |\cdot| $ represents the sup-norm.

\begin{definition}\label{bijiao}
	For two modulus of continuity $ \varpi $ and ${\varpi ^ * }  $, we say that $ {\varpi ^ * } $ is weaker than $ \varpi $ ($ {\varpi ^ * } \gtrsim \varpi  $ or $ {\varpi} \lesssim \varpi ^ *   $ for short), if
	\[\mathop {\overline {\lim } }\limits_{x \to {0^ + }} \frac{{\varpi \left( x \right)}}{{{\varpi ^ * }\left( x \right)}} <  + \infty.\]
	Further, we say that $ {\varpi ^ * } $ is strictly weaker than $ \varpi $ if
	\[\mathop {\overline {\lim } }\limits_{x \to {0^ + }} \frac{{\varpi \left( x \right)}}{{{\varpi ^ * }\left( x \right)}} =0 .\]
\end{definition}
\begin{remark}
	Obviously any modulus of continuity $ \varpi_0  $ is weaker than $ x $ ($ \varpi_0\gtrsim x $). The H\"older type $ \varpi_1\sim x^\alpha $ with any $ 0<\alpha<1 $ is strictly weaker than The Lipschitz type $ \varpi_2\sim x $. The Logarithmic H\"older type $ \varpi_3  \sim {\left( {\log {x^{ - 1}}} \right)^{ - \beta }} $ with any $ \beta>0 $ is strictly weaker than $ \varpi_1 $.
\end{remark}
\begin{remark}
	``Strictly'' implies that $ \varpi_1 $ cannot be accurately characterized by any modulus of continuity of type $ \varpi_2 $, as we have  mentioned  in the introduction.
\end{remark}

\begin{definition}\label{con2.4}
	A function $ f(x) $ is said to be $ C_{k,\varpi} $ with some $ k \in \mathbb{N} $ on $ \mathcal{D} $  ($ f \in C_{k,\varpi} (\mathcal{D}) $ for short), if $ f(x) \in C^k(\mathcal{D}) $ and
	\[\left| {f^{(k)}\left( x \right) - f^{(k)}\left( y \right)} \right| \leqslant {\varpi}\left( {\left| {x - y} \right|} \right),\quad\forall x,y \in \mathcal{D},\; x\ne y .\]
\end{definition}
\begin{remark}\label{666}
	The above definition can be easily extended to the higher dimensional (even infinite dimensional) case. Further	assume that  $ \mathcal{D}  $  is bounded and closed.  If $ f \in C^{k}\left( {\mathcal{D} } \right) $ with some $ k \in \mathbb{N} $, then $ D^{k}f $ automatically has a modulus of continuity $ \varpi $, i.e., $ f $ is $ C_{k,\varpi} $ on $ \mathcal{D} $. However, in the infinite dimensional case, continuity does not imply the existence of modulus of continuity,  because the compactness may be absent at this point. But maybe one could choose some appropriate norms to avoid that.
\end{remark}

\section{Cross-ratio estimates via modulus of continuity} \label{Cross-ratio estimates via modulus of continuity}

Yoccoz \cite{MR0741080} introduced the cross-ratio distortion estimates (asymptotics of double ratios) in dynamical systems for the first time, and proved that  there are no analytic Denjoy counterexamples. See also \cite{MR0997312,Khanin09Invent,MR0968483}. It is worth mentioning that this powerful tool can also be employed to study Denjoy-type inequality under higher smoothness, such as $ C^3 $'s type based on Schwartz derivatives by Teplinsky \cite{MR2461039,Teplinski09}. Objectively speaking, the cross-ratio distortion estimates greatly simplify the conjugacy analysis of circle diffeomorphisms and makes the proof elementary (instead of studying the fundamental segments directly).

We first present some basic notions. Let $ f $ be a strictly increasing function and $ f' $ does not vanish. The \textit{ratio} of three pairwise distinct points $ x_1,x_2,x_3 $ is defined as
\[\mathrm{R}\left( {{x_1},{x_2},{x_3}} \right): = \frac{{{x_1} - {x_2}}}{{{x_2} - {x_3}}},\]
and the \textit{ratio distortion} with respect to of those points and the given $ f $ is
\begin{align*}
	\mathrm{D}\left( {{x_1},{x_2},{x_3};f} \right): &= \frac{{\mathrm{R}\left( {f\left( {{x_1}} \right),f\left( {{x_2}} \right),f\left( {{x_3}} \right)} \right)}}{{\mathrm{R}\left( {{x_1},{x_2},{x_3}} \right)}}\\
	& = \frac{{f\left( {{x_1}} \right) - f\left( {{x_2}} \right)}}{{{x_1} - {x_2}}}:\frac{{f\left( {{x_2}} \right) - f\left( {{x_3}} \right)}}{{{x_2} - {x_3}}}.
\end{align*}
The \textit{cross-ratio} of four pairwise distinct points $ x_1,x_2,x_3,x_4 $ is denoted by
\[\mathrm{Cr}\left( {{x_1},{x_2},{x_3},{x_4}} \right) = \frac{{\left( {{x_1} - {x_2}} \right)\left( {{x_3} - {x_4}} \right)}}{{\left( {{x_2} - {x_3}} \right)\left( {{x_4} - {x_1}} \right)}},\]
and the \textit{cross-ratio distortion} of those points with respect to $ f $ is termed
\[\mathrm{Dist}\left( {{x_1},{x_2},{x_3},{x_4};f} \right) = \frac{{\mathrm{Cr}\left( {f\left( {{x_1}} \right),f\left( {{x_2}} \right),f\left( {{x_3}} \right),f\left( {{x_4}} \right)} \right)}}{{\mathrm{Cr}\left( {{x_1},{x_2},{x_3},{x_4}} \right)}}.\]
One can easily verify that
\[\mathrm{D}\left( {{x_1},{x_2},{x_3};f \circ g} \right) = \mathrm{D}\left( {{x_1},{x_2},{x_3};g} \right) \cdot \mathrm{D}\left( {g\left( {{x_1}} \right),g\left( {{x_2}} \right),g\left( {{x_3}} \right);f} \right),\]
\begin{equation}\label{KhIn2}
	\mathrm{Dist}\left( {{x_1},{x_2},{x_3},{x_4};f} \right) = \frac{{\mathrm{D}\left( {{x_1},{x_2},{x_3};f} \right)}}{{\mathrm{D}\left( {{x_1},{x_4},{x_3};f} \right)}},
\end{equation}
and
\begin{align*}
	&\;\mathrm{Dist}\left( {{x_1},{x_2},{x_3},{x_4};f \circ g} \right)\\
	= &\;\mathrm{Dist}\left( {{x_1},{x_2},{x_3},{x_4};g} \right) \cdot \mathrm{Dist}\left( {g\left( {{x_1}} \right),g\left( {{x_2}} \right),g\left( {{x_3}} \right),g\left( {{x_4}} \right);f} \right).
\end{align*}
In fact, the above definitions are also well defined for the case where two (or three) points are identical, as long as $ \frac{{f\left( x \right) - f\left( x \right)}}{{x - x}}: = f'\left( x \right) $ is defined.

\begin{proposition}[Cross-ratio estimates]\label{lemmaD}
	Assume $ f $ is $ C_{2,\varpi} $ with some modulus of continuity $ \varpi $, and $ f'>0 $ on $ [A,B] $. Then for any $ x_1, x_2, x_3 \in [A,B] $, the following estimate holds:
	\[\mathrm{D}\left( {{x_1},{x_2},{x_3};f} \right) = 1 + \left( {{x_1} - {x_3}} \right)\left( {\frac{{f''}}{{2f'}} + \mathcal{O}\left( {\varpi \left( \Delta  \right)} \right)} \right),\]
	where $ \Delta  = \mathop {\max }\nolimits_{1 \leqslant i \leqslant 3} \left\{ {{x_i}} \right\} - \mathop {\min }\nolimits_{1 \leqslant i \leqslant 3} \left\{ {{x_i}} \right\} $,  and the values of both $ f'' $ and $ f' $ can be taken at any points between $ \mathop {\min }\nolimits_{1 \leqslant i \leqslant 3} \left\{ {{x_i}} \right\} $ and $ \mathop {\max }\nolimits_{1 \leqslant i \leqslant 3} \left\{ {{x_i}} \right\} $.
\end{proposition}
\begin{proof}
	One can trivially verify that the conclusion holds automatically as long as there exists $ x_i=x_j $ with $ i \ne j $, we thus assume that $ x_i\ne x_j $ if $ i \ne j $. Three different cases need to be discussed.  We  first prove a basic case, and the rest of the cases can be directly obtained through it. Let $ x^* $ arbitrarily lie between  $ \mathop {\min }\nolimits_{1 \leqslant i \leqslant 3} \left\{ {{x_i}} \right\} $ and $ \mathop {\max }\nolimits_{1 \leqslant i \leqslant 3} \left\{ {{x_i}} \right\} $, and denote $ \Delta  = \mathop {\max }\nolimits_{1 \leqslant i \leqslant 3} \left\{ {{x_i}} \right\} - \mathop {\min }\nolimits_{1 \leqslant i \leqslant 3} \left\{ {{x_i}} \right\} $.
	
	\noindent\textit{Case 1:} $ x_2 $ lies between $ x_1 $ and $ x_3 $.  By applying the Mean Value Theorem, we obtain (i) $ \zeta_1 $ between $ x_1 $ and $ x_2 $, (ii) $ \zeta_2,\zeta_3 $ between $ x_2 $ and $ x_3 $, (iii) $ \zeta_4 $ between $ \zeta_3 $ and $ x^* $, such that
	\begin{align}
		&\;\frac{{f\left( {{x_1}} \right) - f\left( {{x_2}} \right)}}{{{x_1} - {x_2}}} - \frac{{f\left( {{x_2}} \right) - f\left( {{x_3}} \right)}}{{{x_2} - {x_3}}}\notag\\
		= &\;\left( {f'\left( {{x_2}} \right) + \frac{1}{2}f''\left( {{\zeta _1}} \right)\left( {{x_1} - {x_2}} \right)} \right) - \left( {f'\left( {{x_2}} \right) + \frac{1}{2}f''\left( {{\zeta _2}} \right)\left( {{x_3} - {x_2}} \right)} \right)\notag\\
		= &\;\frac{1}{2}\left( {{x_1} - {x_3}} \right)\left( {f''\left( {{x^ * }} \right) + \left( {f''\left( {{\zeta _1}} \right) - f''\left( {{x^ * }} \right)} \right) + \left( {f''\left( {{\zeta _1}} \right) - f''\left( {{\zeta _2}} \right)} \right)\frac{{{x_3} - {x_2}}}{{{x_1} - {x_3}}}} \right)\notag\\
		\label{Case1d} = &\;\left( {{x_1} - {x_3}} \right)\left( {\frac{1}{2}f''\left( {{x^ * }} \right) + \mathcal{O}\left( {\varpi \left( \Delta  \right)} \right)} \right),
	\end{align}
	and
	\begin{align}
		{\left( {\frac{{f\left( {{x_2}} \right) - f\left( {{x_3}} \right)}}{{{x_2} - {x_3}}}} \right)^{ - 1}} &= \frac{1}{{f'\left( {{\zeta _3}} \right)}} = \frac{1}{{f'\left( {{x^ * }} \right)}}\frac{{f'\left( {{x^ * }} \right)}}{{f'\left( {{\zeta _3}} \right)}}\notag\\
		& = \frac{1}{{f'\left( {{x^ * }} \right)}}\frac{{f'\left( {{\zeta _3}} \right) + f''\left( {{\zeta _4}} \right)\left( {{x^ * } - {\zeta _3}} \right)}}{{f'\left( {{\zeta _3}} \right)}}\notag\\
		\label{Case1dd}& = \frac{1}{{f'\left( {{x^ * }} \right)}}\left( {1 + \mathcal{O}\left( \Delta  \right)} \right).
	\end{align}
	At this point, it follows that
	\begin{align*}
		&\;\mathrm{D}\left( {{x_1},{x_2},{x_3};f} \right) \\
		= &\;1 + \left( {\frac{{f\left( {{x_1}} \right) - f\left( {{x_2}} \right)}}{{{x_1} - {x_2}}} - \frac{{f\left( {{x_2}} \right) - f\left( {{x_3}} \right)}}{{{x_2} - {x_3}}}} \right):\frac{{f\left( {{x_2}} \right) - f\left( {{x_3}} \right)}}{{{x_2} - {x_3}}}\\
		= &\;1 + \left( {{x_1} - {x_3}} \right)\left( {\frac{1}{2}f''\left( {{x^ * }} \right) + \mathcal{O}\left( {\varpi \left( \Delta  \right)} \right)} \right) \cdot \frac{1}{{f'\left( {{x^ * }} \right)}}\left( {1 + \mathcal{O}\left( \Delta  \right)} \right)\\
		= &\;1 + \left( {{x_1} - {x_3}} \right)\left( {\frac{{f''\left( {{x^ * }} \right)}}{{2f'\left( {{x^ * }} \right)}} + \mathcal{O}\left( {\varpi \left( \Delta  \right)} \right)} \right).
	\end{align*}

	\noindent	\textit{Case 2:} $ x_1 $ lies between $ x_2 $ and $ x_3 $.  In view of \eqref{Case1d} and \eqref{Case1dd} in Case 1 we get
	\begin{align*}
		&\;\mathrm{D}\left( {{x_1},{x_2},{x_3};f} \right)\\
		= &\;1 + \left[ {\frac{{{x_1} - {x_3}}}{{{x_2} - {x_3}}}\left( {\frac{{f\left( {{x_2}} \right) - f\left( {{x_1}} \right)}}{{{x_2} - {x_1}}} - \frac{{f\left( {{x_1}} \right) - f\left( {{x_3}} \right)}}{{{x_1} - {x_3}}}} \right)} \right]:\frac{{f\left( {{x_2}} \right) - f\left( {{x_3}} \right)}}{{{x_2} - {x_3}}}\\
		= &\;1 + \left[ {\frac{{{x_1} - {x_3}}}{{{x_2} - {x_3}}}\left( {\left( {{x_2} - {x_3}} \right)\left( {\frac{1}{2}f''\left( {{x^ * }} \right) + \mathcal{O}\left( {\varpi \left( \Delta  \right)} \right)} \right)} \right)} \right] \cdot \frac{{1 + \mathcal{O}\left( \Delta  \right)}}{{f'\left( {{x^ * }} \right)}}\\
		= &\;1 + \left( {{x_1} - {x_3}} \right)\left( {\frac{{f''\left( {{x^ * }} \right)}}{{2f'\left( {{x^ * }} \right)}} + \mathcal{O}\left( {\varpi \left( \Delta  \right)} \right)} \right).
	\end{align*}
	
	\noindent	\textit{Case 3:} $ x_3 $ lies between $ x_1 $ and $ x_2 $. The conclusion holds  by the same argument as in  Case 2 because
	\begin{align*}
		&\;\mathrm{D}\left( {{x_1},{x_2},{x_3};f} \right)\\
		=&\; 1 + \left[ {\frac{{{x_1} - {x_3}}}{{{x_1} - {x_2}}}\left( {\frac{{f\left( {{x_1}} \right) - f\left( {{x_3}} \right)}}{{{x_1} - {x_3}}} - \frac{{f\left( {{x_2}} \right) - f\left( {{x_3}} \right)}}{{{x_2} - {x_3}}}} \right)} \right]:\frac{{f\left( {{x_2}} \right) - f\left( {{x_3}} \right)}}{{{x_2} - {x_3}}}.
	\end{align*}
	
	This completes the proof.
\end{proof}

\begin{proposition}[Cross-ratio distortion estimates]\label{lemmaDist}
	Assume $ f $ is $ C_{2,\varpi} $ with some modulus of continuity $ \varpi $, and $ f'>0 $ on $ [A,B] $. Then for any $ x_1, x_2, x_3, x_4 \in [A,B] $, the following estimate holds:
	\[\mathrm{Dist}\left( {{x_1},{x_2},{x_3},{x_4};f} \right) = 1 + \left( {{x_1} - {x_3}} \right)\mathcal{O}\left( {\varpi \left( \Delta  \right)} \right),\]
	where $ \Delta  = \mathop {\max }\nolimits_{1 \leqslant i \leqslant 4} \left\{ {{x_i}} \right\} - \mathop {\min }\nolimits_{1 \leqslant i \leqslant 4} \left\{ {{x_i}} \right\} $.
\end{proposition}
\begin{proof}
	Recall \eqref{KhIn2}. Then direct calculation gives that
	\begin{align*}
		\mathrm{Dist}\left( {{x_1},{x_2},{x_3},{x_4};f} \right) &= \frac{{\mathrm{D}\left( {{x_1},{x_2},{x_3};f} \right)}}{{\mathrm{D}\left( {{x_1},{x_4},{x_3};f} \right)}}\\
		&= \left( {1 + \left( {{x_1} - {x_3}} \right)\frac{{f''\left( {{x^ * }} \right)}}{{2f'\left( {{x^ * }} \right)}} + \left( {{x_1} - {x_3}} \right)\mathcal{O}\left( {\varpi \left( \Delta  \right)} \right)} \right)\\
		&\;\;\;\;\cdot \left( {1 - \left( {{x_1} - {x_3}} \right)\frac{{f''\left( {{x^ * }} \right)}}{{2f'\left( {{x^ * }} \right)}} + \left( {{x_1} - {x_3}} \right)\mathcal{O}\left( {\varpi \left( \Delta  \right)} \right)} \right)\\
		&= 1 + \left( {{x_1} - {x_3}} \right)\left( {\mathcal{O}\left( {\varpi \left( \Delta  \right)} \right) + \mathcal{O}\left( \Delta  \right) + \mathcal{O}\left( {{\Delta ^2}} \right)} \right)\\
		& = 1 + \left( {{x_1} - {x_3}} \right)\mathcal{O}\left( {\varpi \left( \Delta  \right)} \right).
	\end{align*}
\end{proof}

\section{Circle diffeomorphisms}\label{Circle diffeomorphisms}
\subsection{Notations and lemmas} 
For a given $ \rho  \in {\mathbb{T}^1}\backslash \left\{ 0 \right\} $, the following continued fraction of $ \rho $  is uniquely determined:
\[\rho  = \frac{1}{{{k_1} + \frac{1}{{{k_2} + \frac{1}{{\frac{ \cdots }{{{k_n} + \frac{1}{ \cdots }}}}}}}}}: = \left[ {{k_1},{k_2}, \ldots ,{k_n}, \ldots } \right],\]
where $ k_n $ with $ n \in \mathbb{N}^+ $ are called the partial quotients of $ \rho $. Denote by $ {p_n}/{q_n} = \left[ {{k_1},{k_2}, \ldots ,{k_n}} \right] $ the $ n^\text{th} $ approximant of $ \rho $ (rational approximations), then $ p_n$ and $q_n $ satisfy the following recurrence relations for all $ n \in \mathbb{N} $
\begin{equation}\label{recurrent}
{p_{n + 1}} = {k_{n + 1}}{p_n} + {p_{n - 1}},\quad{q_{n + 1}} = {k_{n + 1}}{q_n} + {q_{n - 1}}
\end{equation}
with $ p_1=1,p_0=0 $ and $ q_1=k_1,q_0=1 $ (here we define $ p_{-1}=1,q_{-1}=0 $ for convenience).

Note that the rotation number $ \rho(T) $ of an orientation-preserving homeomorphism $ T $ on $ \mathbb{T}^1 $ is always well-defined, we therefore could study the conjugacy directly via the partial quotients as well as the approximant of the rotation.  Consider a marked point $ \xi_0 \in \mathbb{T}^1 $ and its trajectory $ {\xi _i} = {T^i}{\xi _0} $ with $ i \in {\mathbb{N}^ + } $, and pick out of it the sequence of the dynamical convergents $ \xi_{q_{n}} $ for $ n \in \mathbb{N} $, indexed by the denominators of the $ n^\text{th} $ approximant $ {p_n}/{q_n} $ of the rotation $ \rho(T)  $. Here we define $ \xi_{q_{-1}}=\xi_0 -1 $ for convenience. The arithmetic properties of rational approximations together with the combinatorial equivalence between $ T $ and the rigid rotation $ R_\rho(T) : \theta \to \theta + \rho(T) \mod 1 $ (i.e., the order of points on the circle for any trajectory coincides with the order of points for $ R_\rho(T) $) show that the dynamical convergents approach the initial point $ \xi_0 $ from both sides:
\begin{equation}\label{jiou}
{\xi _{{q_{ - 1}}}} < {\xi _{{q_1}}} <  \cdots  < {\xi _{{q_{2n + 1}}}} <  \cdots  < {\xi _0} <  \cdots  < {\xi _{{q_{2n}}}} <  \cdots  < {\xi _{{q_2}}} < {\xi _{{q_0}}}.
\end{equation}
In view of \eqref{jiou}, we denote by $ \Delta^n(\xi) $ the $ n ^{\text{th}}$ fundamental segment, where $ \Delta^n(\xi)=\left[ {\xi ,{T^{{q_n}}}\xi } \right] $ if $ n $ is even and $ \Delta^n(\xi)=\left[ {{T^{{q_n}}}\xi ,\xi } \right] $ if $ n $ is odd. For a given marked point $ \xi_0 \in \mathbb{T}^1 $, we denote that $ \Delta_0^{(n)}=\Delta^n(\xi_0) $ and $ \Delta_0^{(i)}=\Delta^n(\xi_i)=T^i\Delta_0^{(n)} $  for all $ i \in \mathbb{N}^+ $. Further, define
\[{l_n}: = {l_n}\left( T \right) = \mathop {\max }\limits_{\xi  \in {\mathbb{T}^1}} \left| {{\Delta ^{\left( n \right)}}\left( \xi  \right)} \right| = \|{T^{{q_n}}} - \mathrm{id}\|_{{{C^0}}}\]
and
\[{\Delta _n}: = {l_n}\left( {{R_\rho }} \right) = \left| {{q_n}\rho  - {p_n}} \right| = {\left( { - 1} \right)^n}\left( {{q_n}\rho  - {p_n}} \right).\]
At this point we have $ l_{-1}=\Delta_{-1}=1 $, and $ l_{n}, \Delta_n \in (0,1)$ for all $ n\in \mathbb{N}^+ $. Recall the recurrence relations \eqref{recurrent} of $ p_n $ and $ q_n $, we arrive at $ {\Delta _n} = {k_{n + 2}}{\Delta _{n + 1}} + {\Delta _{n + 2}} $ for $  - 1 \leqslant n \in \mathbb{Z} $. Since $ \Delta_n $ is strictly monotonically decreasing, we could extend it as a continuous function $ \Delta (\cdot):\left[ {1, + \infty } \right) \to (0,1] $ and denote by $  \Delta^{-1}  $ the inverse of $ \Delta $.

Next we provide several basic but  well-known lemmas  about the combinatorics of trajectories as well as estimates for fundamental segments. Their proofs are classical and can be found in \cite{Khanin09Invent} (see Lemmas 1, 2, 3, 4, 5, 7 respectively), which we omit here for the sake of brevity since they \textit{do not involve the irrationality or the modulus of continuity}.

\begin{lemma}\label{lemmadisjoint}
For any $ \xi \in \mathbb{T}^1 $ and $ 0<i<q_{n+1} $, the segments $ \Delta^{(n)} (\xi)$ and $ \Delta^{(n)} (T^{i}\xi)$ are disjoint (except at the endpoints). In particular, for any fixed $ \xi_0 $, all the segments $ \Delta^{(n)}_{i} $, $ 0 \leqslant i <q_{n+1} $, are disjoint.
\end{lemma}

\begin{lemma}\label{lemmaifenjie}
If a point $ \xi_i $ with some $ i>0 $ belongs to the fundamental segment $ \Delta^{(n)}_{0} $, then $ i $ can be expanded in the form $ i=q_n+ \sum\nolimits_{s = n + 1}^{n + m} {{{\widehat k}_{s + 1}}{q_s}}  $ with some integer $ 0 \leqslant {{\widehat k}_{s + 1}} \leqslant {k_{s + 1}} $, $ n+1 \leqslant s \leqslant n+m$, $m \geqslant 1 $.
\end{lemma}

\begin{lemma}\label{lemmalnxu}
$ l_n \geqslant \Delta_n $.
\end{lemma}

\begin{lemma}
For any $ 0\leqslant j-i<q_{n+1} $, there holds $ \frac{{\left| {\Delta _i^{\left( {n + m} \right)}} \right|}}{{\left| {\Delta _i^{\left( n \right)}} \right|}} \sim \frac{{\left| {\Delta _j^{\left( {n + m} \right)}} \right|}}{{\left| {\Delta _j^{\left( n \right)}} \right|}} $.
\end{lemma}

\begin{lemma}\label{deltal}
$ \frac{{\left| {\Delta _0^{\left( {n + m} \right)}} \right|}}{{\left| {\Delta _0^{\left( n \right)}} \right|}} = \mathcal{O}\left( {\frac{{{l_{n + m}}}}{{{l_n}}}} \right) $.
\end{lemma}

\begin{lemma}\label{lzhishu}
$ \frac{{{l_{n + m}}}}{{{l_n}}} = \mathcal{O}\left( {{\lambda ^m}} \right) $, where the constant $ \lambda \in (\frac{1}{2},1)$ is defined in the classical Denjoy theory below, see Statement (B).
\end{lemma}

\subsection{Classical Denjoy's theory}\label{Subdenjoy}
Here we go back to the classical Denjoy theory \cite{A.Denjoy} which will be used later, see details from Lecture 10 (Homeomorphisms and Diffeomorphisms of the Circle) in the book \cite{Sinaibook}.

For any orientation-preserving circle diffeomorphism $ T \in C^{1+{\rm BV}} (\overline{\mathbb{T}^1})$ ($ {\rm BV} $ stands for bounded variation) with any irrational rotation number $ \rho $ (i.e., need not satisfy any nonresonant conditions such as the Diophantine's type), the following estimates hold:

\begin{enumerate}[label=(\Alph*)] 
\item \label{item:a} $ \log(T^{q_n})'(\xi_0) =\mathcal{O}(1)$;

\item \label{item:b} There exists $ \lambda \in (\frac{1}{2},1) $ such that $ \frac{{\left| {\Delta _0^{\left( {n + m} \right)}} \right|}}{{\left| {\Delta _0^{\left( n \right)}} \right|}} = \mathcal{O}\left( {{\lambda ^m}} \right) $. More precisely, we can write $ \lambda $ explicitly:
\[\lambda  = \frac{1}{{\sqrt {1 + {e^{ - \mathcal{C}}}} }},\quad\mathcal{C} = \int_0^1 {\left| {\frac{d}{{dx}}\log T'\left( x \right)} \right|dx};\]

\item \label{item:c} There exists a \textit{homeomorphism} $ \phi $ such that conjugates $ T $ to $ R_{\rho} $, i.e., $ \phi  \circ T \circ {\phi ^{ - 1}} = {R_\rho } $. This is called the topological (or continuous) equivalence of $ T $ and $ R_\rho $.
\end{enumerate}

\subsection{Denjoy-type inequality via modulus of continuity}\label{DenjoyDenjoy}

To establish the generalized Denjoy-type inequality in the case of modulus of continuity beyond the H\"older type, we need some basic lemmas based on the cross-ratio distortion estimates we introduced previously.

\begin{lemma}
For any fixed $ \xi_0 $, define functions:
\begin{equation}\label{MnKndy}
	\left\{ \begin{aligned}
		&{M_n}\left( \xi  \right) := \mathrm{D}\left( {{\xi _0},\xi ,{\xi _{{q_{n - 1}}}};{T^{{q_n}}}} \right),\ &\xi  \in \Delta _0^{\left( {n - 1} \right)} \hfill, \\
		&{K_n}\left( \xi  \right) := \mathrm{D}\left( {{\xi _0},\xi ,{\xi _{{q_n}}};{T^{{q_{n - 1}}}}} \right),\ &\xi  \in \Delta _0^{\left( {n - 2} \right)} \hfill. \\
	\end{aligned}  \right.
\end{equation}
Then we have the following equations:
\begin{align}
	\label{eq1}{M_n}\left( {{\xi _0}} \right) \cdot {M_n}\left( {{\xi _{{q_{n - 1}}}}} \right) &= {K_n}\left( {{\xi _0}} \right) \cdot {M_n}\left( {{\xi _{{q_n}}}} \right), \\
	\label{eq2}{K_{n + 1}}\left( {{\xi _{{q_{n - 1}}}}} \right) - 1 &= \frac{{\left| {\Delta _0^{\left( {n + 1} \right)}} \right|}}{{\left| {\Delta _0^{\left( {n - 1} \right)}} \right|}}\left( {{M_n}\left( {{\xi _{{q_{n + 1}}}}} \right) - 1} \right), \\
	\label{eq3}\frac{{{{\left( {{T^{{q_n}}}} \right)}^\prime }\left( {{\xi _0}} \right)}}{{{M_n}\left( {{\xi _0}} \right)}} - 1 &= \frac{{\left| {\Delta _0^{\left( n \right)}} \right|}}{{\left| {\Delta _0^{\left( {n - 1} \right)}} \right|}}\left( {1 - \frac{{{{\left( {{T^{{q_{n - 1}}}}} \right)}^\prime }\left( {{\xi _0}} \right)}}{{{K_n}\left( {{\xi _0}} \right)}}} \right).
\end{align}
\end{lemma}
\begin{proof}
Through the definitions in \eqref{MnKndy} we arrive at the following expressions:
\begin{align*}
	{M_n}\left( {{\xi _0}} \right) &= \mathrm{D}\left( {{\xi _0},{\xi _0},{\xi _{{q_{n - 1}}}};{T^{{q_n}}}} \right) = ({T^{{q_n}}})'\left( {{\xi _0}} \right):\frac{{\left| {\Delta _{{q_n}}^{\left( {n - 1} \right)}} \right|}}{{\left| {\Delta _0^{\left( {n - 1} \right)}} \right|}},\\
	{M_n}\left( {{\xi _{{q_{n - 1}}}}} \right) &= \mathrm{D}\left( {{\xi _0},{\xi _{{q_{n - 1}}}},{\xi _{{q_{n - 1}}}};{T^{{q_n}}}} \right) = \frac{{\left| {\Delta _{{q_n}}^{\left( {n - 1} \right)}} \right|}}{{\left| {\Delta _0^{\left( {n - 1} \right)}} \right|}}:({T^{{q_n}}})'\left( {{\xi _{{q_{n - 1}}}}} \right),\\
	{M_n}\left( {{\xi _{{q_{n + 1}}}}} \right) &= \mathrm{D}\left( {{\xi _0},{\xi _{{q_{n + 1}}}},{\xi _{{q_{n - 1}}}};{T^{{q_n}}}} \right) = \frac{{\left| {\Delta _{{q_n}}^{\left( {n + 1} \right)}} \right|}}{{\left| {\Delta _0^{\left( {n + 1} \right)}} \right|}}:\frac{{\left| {\Delta _{{q_n}}^{\left( {n - 1} \right)}} \right| - \left| {\Delta _{{q_n}}^{\left( {n + 1} \right)}} \right|}}{{\left| {\Delta _0^{\left( {n - 1} \right)}} \right| - \left| {\Delta _0^{\left( {n + 1} \right)}} \right|}},\\
	{K_n}\left( {{\xi _0}} \right) &= \mathrm{D}\left( {{\xi _0},{\xi _0},{\xi _{{q_n}}};{T^{{q_{n - 1}}}}} \right) = ({T^{{q_{n - 1}}}})'\left( {{\xi _0}} \right):\frac{{\left| {\Delta _{{q_{n - 1}}}^{\left( n \right)}} \right|}}{{\left| {\Delta _0^{\left( n \right)}} \right|}},\\
	{K_n}\left( {{\xi _{{q_n}}}} \right) &= \mathrm{D}\left( {{\xi _0},{\xi _{{q_n}}},{\xi _{{q_n}}};{T^{{q_{n - 1}}}}} \right) = \frac{{\left| {\Delta _{{q_{n - 1}}}^{\left( n \right)}} \right|}}{{\left| {\Delta _0^{\left( n \right)}} \right|}}:({T^{{q_{n - 1}}}})'\left( {{\xi _{{q_n}}}} \right),\\
	{K_{n + 1}}\left( {{\xi _{{q_{n - 1}}}}} \right) &= \mathrm{D}\left( {{\xi _0},{\xi _{{q_{n - 1}}}},{\xi _{{q_{n + 1}}}};{T^{{q_n}}}} \right) = \frac{{\left| {\Delta _{{q_n}}^{\left( {n - 1} \right)}} \right|}}{{\left| {\Delta _0^{\left( {n - 1} \right)}} \right|}}:\frac{{\left| {\Delta _{{q_n}}^{\left( {n - 1} \right)}} \right| - \left| {\Delta _{{q_n}}^{\left( {n + 1} \right)}} \right|}}{{\left| {\Delta _0^{\left( {n - 1} \right)}} \right| - \left| {\Delta _0^{\left( {n + 1} \right)}} \right|}}.
\end{align*}
Further, one notices that
\[({T^{{q_{n - 1}}}})'\left( {{\xi _{{q_n}}}} \right) \cdot ({T^{{q_n}}})'\left( {{\xi _0}} \right) = ({T^{{q_n}}})'\left( {{\xi _{{q_{n - 1}}}}} \right) \cdot ({T^{{q_{n - 1}}}})'\left( {{\xi _0}} \right),\]
and
\[\left| {\Delta _0^{\left( {n - 1} \right)}} \right| + \left| {\Delta _0^{\left( n \right)}} \right| = \left| {\Delta _{{q_n}}^{\left( {n - 1} \right)}} \right| + \left| {\Delta _{{q_{n - 1}}}^{\left( n \right)}} \right|,\]
then the conclusions \eqref{eq1}, \eqref{eq2} and \eqref{eq3} can be directly derived.
\end{proof}

\begin{lemma}\label{logdistjie}
\begin{equation}\notag
	\left\{ \begin{aligned}
		&\log \mathrm{Dist}\left( {{\xi _0},\xi ,{\xi _{{q_{n - 1}}}},\eta :{T^{{q_n}}}} \right) = \mathcal{O}\left( {\varpi \left( {{l_{n - 1}}} \right)} \right),&\xi ,\eta  \in \Delta _0^{\left( {n - 1} \right)} \hfill, \\
		&\log \mathrm{Dist}\left( {{\xi _0},\xi ,{\xi _{{q_n}}},\eta :{T^{{q_{n - 1}}}}} \right) = \mathcal{O}\left( {\varpi \left( {{l_n}} \right)} \right),&\xi ,\eta  \in \Delta _0^{\left( {n - 2} \right)} \hfill. \\
	\end{aligned}  \right.
\end{equation}
\end{lemma}
\begin{proof} The result follows from Proposition \ref{lemmaDist}, Lemma \ref{lemmadisjoint}, and Lemma \ref{lemmaifenjie}. The proof is identical to that of Lemma 6 in \cite{Khanin09Invent}. 
\end{proof}

Now we are in a position to establish the crucial Denjoy-type inequality via modulus of continuity, i.e., a stronger version of Statement \ref{item:a}. It can be seen later that it plays a significant role in dealing with the $ C^1 $ smoothness of the homeomorphism $ \phi $. Actually, considering the work of Teplinski \cite{Teplinski09}, the result below could be extended to the case of $ C_{3,\varpi} $ in studying $ C^2 $ conjugacy, we do not pursue that.

\begin{theorem}[Denjoy-type inequality via modulus of continuity]\label{denjoymodulus}
Assume $ T $ to be a $ C_{2,\varpi} $ orientation-preserving circle diffeomorphism with an irrational rotation number (need not satisfy any nonresonant conditions) and a modulus of continuity $ \varpi $. Then 
\[{({T^{{q_n}}})^\prime }(\xi ) = 1 + \mathcal{O}\left( {{\tau _n}} \right),\quad {\tau _n}: = \sum\limits_{k = 0}^n {\frac{{{l_n}}}{{{l_{n - k}}}}\varpi \left( {{l_{n - k - 1}}} \right)}  .\]
In particular,
\begin{equation}\label{taunjie}
	{\tau _n} = \mathcal{O}\left( {{\lambda ^n}\int_{{\lambda ^n}}^1 {{y^{ - 2}}\varpi \left( y \right)dy} } \right)=o(1).
\end{equation}
\end{theorem}
\begin{remark}
In fact, the regularity of $ T $ only needs to be $ C^2 $,  because there automatically exists  a  modulus of continuity $ \varpi $ such that $ T \in C_{2,\varpi}({\mathbb{T}^1}) $ due to Remark \ref{666}. Conclusion  \eqref{taunjie} implies that we do not have any additional restrictions for the above $ \varpi $, which is different from Theorem \ref{varpitype}, etc.
\end{remark}
\begin{remark}
Direct calculation gives that $ {\tau _n} = \mathcal{O}\left( {{\lambda ^{\alpha n}}} \right) $ if $ \varpi_1 \left( x \right) \sim {x^\alpha } $ with some $ 0 < \alpha  < 1 $ (the H\"older type), and $ {\tau _n} = \mathcal{O}\left( {n{\lambda ^n}} \right) $ if $ \varpi_2 \left( x \right) \sim x $ (the Lipschitz type).  In both cases $ \tau_n $ decreases exponentially, and they are  estimated to be optimal, as commented in \cite{Khanin09Invent}. As to the Logarithmic H\"older type $ \varpi_3 \sim \ (\log{x^{-1}})^{-1-\varepsilon} $ with $ \varepsilon>0 $ that satisfies the Dini condition \eqref{DINI}, $ {\tau _n} = \mathcal{O}\left( {{n^{ - 1 - \varepsilon }}} \right) $ (the same as that in \cite{ETDScon}) due to the asymptotic behavior $ \int_2^X {{{\left( {\log z} \right)}^{ - 1 - \varepsilon }}dz}  \sim X{\left( {\log X} \right)^{ - 1 - \varepsilon }} $ for $ X $ sufficiently large.
\end{remark}
\begin{proof}
It follows from Lemma \ref{logdistjie} that
\begin{align}
	\frac{{{M_n}\left( \xi  \right)}}{{{M_n}\left( \eta  \right)}} &= \mathrm{Dist}\left( {{\xi _0},\xi ,{\xi _{{q_{n - 1}}}},\eta ;{T^{{q_n}}}} \right) = 1 + \mathcal{O}\left( {\varpi \left( {{l_{n - 1}}} \right)} \right),\notag \\
	\frac{{{K_n}\left( \xi  \right)}}{{{K_n}\left( \eta  \right)}} &= \mathrm{Dist}\left( {{\xi _0},\xi ,{\xi _{{q_n}}},\eta ;{T^{{q_{n - 1}}}}} \right) = 1 + \mathcal{O}\left( {\varpi \left( {{l_n}} \right)} \right).\notag
\end{align}
Recall Statement \ref{item:a}, i.e., $ \log(T^{q_n})'(\xi_0) =\mathcal{O}(1)$, then there exist $ c_1, c_2 \in \mathbb{R} $ independent of $ n $ such that $ {c_1} \leqslant {M_n},{K_n} \leqslant {c_2} $, which gives that
\begin{align}
	\label{Mn6}{M_n}\left( \xi  \right) &= {m_n} + \mathcal{O}\left( {\varpi \left( {{l_{n - 1}}} \right)} \right),\\
	\label{Kn6}{K_n}\left( \xi  \right) &= {m_n} + \mathcal{O}\left( {\varpi \left( {{l_n}} \right)} \right),
\end{align}
where $ m_n^2 $ is the products in \eqref{eq1},  i.e.,
\[{m_n} = \sqrt {{M_n}\left( {{\xi _0}} \right) \cdot {M_n}\left( {{\xi _{{q_{n - 1}}}}} \right)}  = \sqrt {{K_n}\left( {{\xi _0}} \right) \cdot {K_n}\left( {{\xi _{{q_n}}}} \right)} .\]
Through \eqref{eq2}, \eqref{Mn6}, \eqref{Kn6} and Lemma \ref{deltal}, one can verify that
\[{m_{n + 1}} - 1 = \frac{{\left| {\Delta _0^{\left( {n + 1} \right)}} \right|}}{{\left| {\Delta _0^{\left( {n - 1} \right)}} \right|}}\left( {{m_n} - 1} \right) + \mathcal{O}\left( {\varpi \left( {{l_{n + 1}}} \right)} \right),\]
which leads to
\begin{align}
	{m_n} - 1 &= \mathcal{O}\left( {\sum\limits_{k = 0}^n {\varpi \left( {{l_{n - k}}} \right)\frac{{\left| {\Delta _0^{\left( n \right)}} \right|\left| {\Delta _0^{\left( {n - 1} \right)}} \right|}}{{\left| {\Delta _0^{\left( {n - k} \right)}} \right|\left| {\Delta _0^{\left( {n - k - 1} \right)}} \right|}}} } \right)\notag \\
	& = \mathcal{O}\left( {\sum\limits_{k = 0}^n {\varpi \left( {{l_{n - k}}} \right)\frac{{{l_n}}}{{{l_{n - k}}}} \cdot \frac{{{l_{n - 1}}}}{{{l_{n - k - 1}}}}} } \notag \right)\\
	& = \mathcal{O}\left( {\varpi \left( {{l_n}} \right)\sum\limits_{k = 0}^n {\frac{{\zeta \left( {{l_n}} \right)}}{{\zeta \left( {{l_{n - k}}} \right)}}{\lambda ^k}} } \right)\notag \\
	& = \mathcal{O}\left( {\varpi \left( {{l_n}} \right)\sum\limits_{k = 0}^n {{\lambda ^k}} } \right)\notag \\
	\label{Proof3-3}& = \mathcal{O}\left( {\varpi \left( {{l_n}} \right)} \right),
\end{align}
where $ \zeta \left( x \right): = x/\varpi \left( x \right) $, and we assume that $ \zeta $ is  monotonically increasing without loss of generality since $ \varpi $ is a  modulus of continuity, hence $ \zeta \left( {{l_n}} \right) \leqslant \zeta \left( {{l_{n - k}}} \right) $ for sufficiently large $ n,k $ thanks to $ {l_n} = \mathcal{O}\left( {{\lambda ^k}{l_{n - k}}} \right) \leqslant {l_{n - k}} $.

Substituting \eqref{Proof3-3} into \eqref{Mn6} and \eqref{Kn6}, we arrive at
\begin{align}
	{M_n}\left( \xi  \right) &= 1 + \mathcal{O}\left( {\varpi \left( {{l_n}} \right)} \right) + \mathcal{O}\left( {\varpi \left( {{l_{n - 1}}} \right)} \right) \notag \\
	\label{Mn7}& = 1 + \mathcal{O}\left( {\varpi \left( {{l_{n - 1}}} \right)} \right),
\end{align}
and
\begin{align}
	{K_n}\left( \xi  \right) &= 1 + \mathcal{O}\left( {\varpi \left( {{l_n}} \right)} \right) + \mathcal{O}\left( {\varpi \left( {{l_n}} \right)} \right) \notag \\
	\label{Kn7}& = 1 + \mathcal{O}\left( {\varpi \left( {{l_n}} \right)} \right).
\end{align}

Substituting \eqref{Mn7} and \eqref{Kn7} into \eqref{eq3}, we derive
\begin{align}
	&{\left( {{T^{{q_{n + 1}}}}} \right)^\prime }\left( {{\xi _0}} \right)\left( {1 + \mathcal{O}\left( {\varpi \left( {{l_n}} \right)} \right)} \right) - 1\notag \\
	\label{zhankai}  = & \frac{{\left| {\Delta _0^{\left( {n + 1} \right)}} \right|}}{{\left| {\Delta _0^{\left( n \right)}} \right|}}\left[ {1 - {{\left( {{T^{{q_n}}}} \right)}^\prime }\left( {{\xi _0}} \right)\left( {1 + \mathcal{O}\left( {\varpi \left( {{l_{n + 1}}} \right)} \right)} \right)} \right].
\end{align}
In view of Statements \ref{item:a} and \ref{item:b}, we get
\[\frac{{\left| {\Delta _0^{\left( {n + 1} \right)}} \right|}}{{\left| {\Delta _0^{\left( n \right)}} \right|}} = \mathcal{O}\left( {\frac{{{l_{n + 1}}}}{{{l_n}}}} \right) = \mathcal{O}\left( 1 \right),\quad {\left( {{T^{{q_n}}}} \right)^\prime }\left( {{\xi _0}} \right) = \mathcal{O}\left( 1 \right),\]
which leads to
\begin{equation}\label{Tqn}
	{\left( {{T^{{q_{n + 1}}}}} \right)^\prime }\left( {{\xi _0}} \right) - 1 = \frac{{\left| {\Delta _0^{\left( {n + 1} \right)}} \right|}}{{\left| {\Delta _0^{\left( n \right)}} \right|}}\left( {1 - {{\left( {{T^{{q_n}}}} \right)}^\prime }\left( {{\xi _0}} \right)} \right) + \mathcal{O}\left( {\varpi \left( {{l_n}} \right)} \right)
\end{equation}
because of \eqref{zhankai}. By iterating \eqref{Tqn}, we get
\begin{align}
	{\left( {{T^{{q_n}}}} \right)^\prime }\left( {{\xi _0}} \right) &= 1 + \mathcal{O}\left( {\sum\limits_{k = 0}^n {\frac{{\left| {\Delta _0^{\left( n \right)}} \right|}}{{\left| {\Delta _0^{\left( {n - k} \right)}} \right|}}\varpi \left( {{l_{n - k - 1}}} \right)} } \right)\notag \\
	\label{Proof3-4} &= 1 + \mathcal{O}\left( {\sum\limits_{k = 0}^n {\frac{{{l_n}}}{{{l_{n - k}}}}\varpi \left( {{l_{n - k - 1}}} \right)} } \right) \\
	& = 1 + \mathcal{O}\left( {{\tau _n}} \right),\notag
\end{align}
where Lemma \ref{deltal} is used in \eqref{Proof3-4}, and
\[{\tau _n}: = \sum\limits_{k = 0}^n {\frac{{{l_n}}}{{{l_{n - k}}}}\varpi \left( {{l_{n - k - 1}}} \right)} .\]

As to \eqref{taunjie}, we first assume that $ {x^{ - 2}}\varpi \left( x \right) $ is decreasing on $ (0,1) $ without loss of generality due to Definition \ref{definition1}, therefore by applying Lemma \ref{lzhishu} we obtain that
\begin{align*}
	{\tau _n}&  = \sum\limits_{k = 0}^n {\frac{{{l_n}}}{{{l_{n - k}}}}\varpi \left( {{l_{n - k - 1}}} \right)} \\
	&  = \mathcal{O}\left( {\sum\limits_{k = 0}^n {{\lambda ^k}\varpi \left( {{\lambda ^{n - k}}} \right)} } \right)\\
	& = \mathcal{O}\left( {{\lambda ^n}\sum\limits_{k = 0}^n {\left( {{\lambda ^k} - {\lambda ^{k + 1}}} \right) \cdot {\lambda ^{ - k}}\varpi \left( {{\lambda ^k}} \right)} } \right)\\
	& = \mathcal{O}\left( {{\lambda ^n}\int_{{\lambda ^n}}^1 {{y^{ - 2}}\varpi \left( y \right)dy} } \right).
\end{align*}
Additionally, it can be proved according to L'Hospital's rule and $ \varpi \left( {0 + } \right) = 0 $ that
\begin{align*}
	\mathop {\lim }\limits_{n \to  + \infty } {\lambda ^n}\int_{{\lambda ^n}}^1 {{y^{ - 2}}\varpi \left( y \right)dy}  & = \mathop {\lim }\limits_{z \to  + \infty } \frac{{\int_{{z^{ - 1}}}^1 {{y^{ - 2}}\varpi \left( y \right)dy} }}{z}\\
	& = \mathop {\lim }\limits_{z \to  + \infty } \varpi \left( {{z^{ - 1}}} \right) = 0,
\end{align*}
which gives \eqref{taunjie}.

This completes the proof of Theorem \ref{denjoymodulus}.
\end{proof}

\subsection{$ C^1 $ conjugacy through the convergence of $ \sum\nolimits_{n = 0}^\infty  {{k_{n + 1}}{\tau _n}}  $}
To obtain the $ C^1 $ regularity of $ \phi $ in Statement \ref{item:c}, we employ the method of constructing the continuous density $ h: \mathbb{T}^1 \to \mathbb{R}^+ $ of the invariant probability measure for $ T $  and the corresponding homological equation, see for instance,  \cite{Arnold,Khanin09Invent,MR0610981,SinaiKhanin89}. \textit{Specifically, the convergence of $ \sum\nolimits_{n = 0}^\infty  {{k_{n + 1}}{\tau _n}}  $ implies  $ C^1 $ conjugacy.}

\begin{theorem}\label{c1conjugacy}
The conjugation $ \phi $ in Statement \ref{item:c} is $ C^1 $ if
\begin{equation}\label{shoulian}
	\sum\limits_{n = 0}^\infty  {{k_{n + 1}}{\tau _n}}< \infty.
\end{equation}
\end{theorem}
\begin{proof}
Let $ \xi_0 $ be arbitrarily fixed and consider the corresponding trajectory $ \Theta : = \left\{ {{T^i}{\xi _0}:i \in \mathbb{N}} \right\} $. If $ \Theta $ fills up the circle in a dense way by ergodicity, then the homeomorphism $ \phi $ in the classical Denjoy's theory can be defined uniquely up to a rotation $ \rho $, and the conjugation $ \phi $ can be written as
\[\phi \left( \xi  \right) = \int_{{\xi _0}}^\xi  {\mu \left( {dx} \right)}, \]
where $ \mu \left( {dx} \right) $ is a normalized measure invariant with respect to $ T $.

Firstly, define a discrete mapping $ \gamma $ on $ \Theta $ that depends on $ \xi_0 $:
\begin{equation}\label{gammaxunfu}
	\gamma \left( {{\xi _0}} \right) = 0,\quad \gamma \left( {{\xi _{i + 1}}} \right) - \gamma \left( {{\xi _i}} \right) =  - \log T'\left( {{\xi _i}} \right),\;i \in {\mathbb{N}^ + }.
\end{equation}
Then we obtain the following by Lemma \ref{lemmaifenjie}, Theorem \ref{denjoymodulus} and Cauchy's theorem through the convergence in \eqref{shoulian}, as long as $ {\xi _j} \in \Delta _i^{\left( n \right)} $ with $ j > i $:
\begin{align*}
	\left| {\gamma \left( {{\xi _j}} \right) - \gamma \left( {{\xi _i}} \right)} \right| &\leqslant \sum\limits_{s = n}^\infty  {{k_{s + 1}}\left| {\log T'\left( {{\xi _s}} \right)} \right|} \\
	&=\mathcal{O}\left( {\sum\limits_{s = n}^\infty  {{k_{s + 1}}{\tau _s}} } \right) = o\left( 1 \right),\quad n \to  + \infty .
\end{align*}
This implies that $ \gamma  \in C\left( \Theta  \right) $. Noting the arbitrary choice of $ \xi_0 $ and the density of $ \Theta $, we can continuously extend the mapping $ \gamma $ onto $ \mathbb{T}^1 $. Now define the density $ h $ as:
\[h\left( \xi  \right) = \frac{{{e^{\gamma \left( \xi  \right)}}}}{{\int_{{\mathbb{T}^1}} {{e^{\gamma \left( x \right)}}dx} }}.\]
Obviously $ h>0 $ is continuous on $ \mathbb{T}^1 $, and
\[\int_{{\mathbb{T}^1}} {h\left( \xi  \right)d\xi }  = 1,\]
which implies that $ h $ is indeed a normalized invariant measure, and
\begin{equation}\label{homo1}
	\gamma \left( {T\xi } \right) - \gamma \left( \xi  \right) =  - \log T'\left( \xi  \right)
\end{equation}
by \eqref{gammaxunfu}. Note that by \eqref{homo1} we obtain the desired homological equation
\begin{equation}\label{homo2}
	h\left( {T\xi } \right) = \frac{{{e^{\gamma \left( {T\xi } \right)}}}}{{\int_{{\mathbb{T}^1}} {{e^{\gamma \left( x \right)}}dx} }} = \frac{{{e^{\gamma \left( \xi  \right) - \log T'\left( \xi  \right)}}}}{{\int_{{\mathbb{T}^1}} {{e^{\gamma \left( x \right)}}dx} }} = \frac{{h\left( \xi  \right)}}{{T'\left( \xi  \right)}}.
\end{equation}
Define a mapping
\[\phi \left( \xi  \right): = \int_{{\xi _0}}^\xi  {h\left( x \right)dx}  \equiv \int_{{\xi _0}}^\xi  {\mu \left( {dx} \right)} .\]
One can verify that $ \phi $ is indeed a $ C^1 $ diffeomorphism such that the following conjugate holds
\begin{equation}\label{weifengonge}
	\phi  \circ T \circ {\phi ^{ - 1}} = {R_\rho },
\end{equation}
because through the homological equation \eqref{homo2} we have
\[{\left( {\phi  \circ T} \right)^\prime }  = \phi '\left( {T\xi } \right) \cdot T'\left( \xi  \right) = h\left( {T\xi } \right) \cdot T'\left( \xi  \right) = h\left( \xi  \right) = {\left( {{R_\rho } \circ \phi } \right)^\prime }.\]
This completes the proof.
\end{proof}

\subsection{$ C^1 $ conjugacy and further regularity through irrational rotation $ \rho $ of $ \varphi $-type}\label{SubHerman}
Let $ \rho  = \left[ {{k_1},{k_2}, \ldots } \right] \in \mathbb{T}^1  $ be an irrational number. The simplest case is where $ \rho $ is of the constant type, or equivalently, is termed Diophantine of exponent $ 2 $, if $ k_n=\mathcal{O}(1) $, as pointed out by Petersen in \cite{Acta96}. This situation appears  in many fields, including dynamical systems, number theory, asymptotic analysis, and even stochastic fields, because its good properties can lead to some remarkable results, such as the Hardy-Littlewood series $ \sum\nolimits_{n = 1}^\infty  {{{\left( {n\sin \left( {\pi n\rho } \right)} \right)}^{ - 1}}}   $ and extensions \cite{MR3308897}. A more complicated case could be considered as $ k_n=\mathcal{O}(n^{\nu}) $ with some $ \nu>0 $, i.e., the polynomial type. These two  situations have attracted a lot of research in conjugacy of circle diffeomorphisms, see for instance, Herman's Theorem \cite{Herman79}, i.e., a $ C^{2+\gamma} $ ($ \gamma>0 $) mapping $ T $ in Theorem \ref{denjoymodulus} with a polynomial type irrational rotation $ \rho $ can lead to $ C^1 $ conjugacy in Denjoy's theory.   See also Sina\u{\i} and Khanin   \cite{cmp87,SinaiKhanin89} for irrationality of the polynomial type. This is somewhat different in characterizing irrationality  from Diophantine, Bruno and even Liouville conditions, but they are obviously more convenient and straightforward for studying rotational $ C^1 $ conjugacy than the latter, we thus apply the Denjoy-type inequality from this point of view. In addition, for some special irrational numbers, their partial quotients are very characteristic, such as the constant type
\[\phi  = \frac{{\sqrt 5  - 1}}{2} = \left[ {1,1,1, \ldots } \right]\quad (\text{Golden Number})\]
and 
\[\sqrt 2  - 1 = \left[ {2,2,2, \ldots } \right].\]
More generally, for any $ k\in \mathbb{N}^+ $, one can trivially verify that
\[\frac{{\sqrt {{k^2} + 4}  - k}}{2} = \left[ {k,k,k, \ldots } \right] \in \left\{ {x:{x^2} + kx - 1 = 0} \right\}.\]
As to the polynomial type, see for instance, 
\[\tanh 1 = \frac{{e - 1}}{{e + 1}} = \left[ {1,3,5,7, \ldots } \right]\]
by Gauss in 1812, $ k_n=\mathcal{O}(n) $ at this point. To highlight the conjugacy of most (almost) irrational rotations, one has to show whether irrational numbers are many in the sense of Lebesgue measure (or full measure) in the perspective of direct study of partial quotients. Fortunately for almost all irrational numbers, their partial quotients satisfy $ k_n=\mathcal{O}(n^\nu) $ with any $ \nu>0 $, see Lemma \ref{Knlemma} in Appendix (or see \cite[Corollary A2.2]{ETDScon} for more precise explicit examples). Although this is sufficient to represent almost irrational numbers, it is still of interest to study the more general cases of irrationality, where our regularity requirements for the mapping $ T $ are different in dealing with $ C^1 $ conjugacy.  Namely, for a nondecreasing continuous function $ \varphi $ on $ \mathbb{R}^+ $, we say that an irrational number $ \rho\in \mathbb{T}^1 $ is of $ \varphi $-type  if its partial quotients satisfy $ k_{n+1}=\mathcal{O}(\varphi(n)) $. Through  the Denjoy inequality via modulus of continuity established previously, we could give a criterion on the required regularity for $ C^1 $ conjugacy which is different from known results, thanks to the cross-ratio  distortion estimates. 
They not only simplify the proof and derive the optimal results of the H\"older type, but also bring the new perspective in the sense of weaker  continuity below.  

Next, we present a sufficient integrability condition on regularity of $ T $ for $ C^1 $ conjugacy based on the general $\varphi $-type irrationality, and provide several explicit examples. In other words, there exist different regularity requirements for $ T $ when $ k_{n+1} $ has different $ \varphi $-forms. Due to the optimality of Denjoy-type inequality claimed in \cite{Khanin09Invent}, our integrability criterion might also be accurate since we cover the case of the H\"older type.

\begin{theorem}[Main Theorem]\label{varpitype}
Let $ T $ be a $ C_{2,\varpi} $  orientation-preserving circle diffeomorphism with	rotation number $ \rho=[k_1,k_2, \ldots ] $, $ k_{n+1}=\mathcal{O}(\varphi(n)) $. Then the conjugation $ \phi $ in \eqref{weifengonge} is $ C^1 $ if the following integrability condition holds:
\begin{equation}\label{jifenyouxian}
	\int_0^1 {\left( {\int_0^y {\varphi \left( {{{\log }_\lambda }x} \right)dx} } \right)\frac{{\varpi \left( y \right)}}{{{y^2}}}dy}  <  + \infty ,
\end{equation}
where $ 0<\lambda<1 $ is given in Statement \ref{item:b}.
\end{theorem}
\begin{remark}\label{remark999}
Note that $ C^1 $ conjugacy can be directly obtained from Theorems \ref{denjoymodulus} and \ref{c1conjugacy}, as long as
\begin{equation}\label{414}
	\sum\limits_{n = 0}^\infty  {{k_{n + 1}}{\lambda ^n}\int_0^{{\lambda ^n}} {{y^{ - 2}}\varpi \left( y \right)dy} }  = \mathcal{O}\left( 1 \right).
\end{equation}
We still establish the $ \varphi $-type theorem because the convergence condition is completely explicit (we suspect it is almost sharp in  general settings), although it may not be accurate for some given irrational numbers, since $ k_n $ may be large only at a very small number of points (although they are infinite many). We do not pursue that.
\end{remark}
\begin{proof}
In view of \eqref{jifenyouxian}, we derive that
\begin{align*}
	\sum\limits_{n = 0}^\infty  {{k_{n + 1}}{\tau _n}}  &= \mathcal{O}\left( {\sum\limits_{n = 0}^\infty  {\varphi \left( n \right){\tau _n}} } \right)\\
	& = \mathcal{O}\left( {\sum\limits_{n = 0}^\infty  {\varphi \left( n \right){\lambda ^n}\int_{{\lambda ^n}}^1 {\frac{{\varpi \left( y \right)}}{{{y^2}}}dy} } } \right)\\
	& = \mathcal{O}\left( {\sum\limits_{n = 0}^\infty  {\left( {{\lambda ^n} - {\lambda ^{n + 1}}} \right)\varphi \left( {{{\log }_\lambda }{\lambda ^n}} \right)\int_{{\lambda ^n}}^1 {\frac{{\varpi \left( y \right)}}{{{y^2}}}dy} } } \right)\\
	& = \mathcal{O}\left( {\int_0^1 {\varphi \left( {{{\log }_\lambda }x} \right)dx\int_x^1 {\frac{{\varpi \left( y \right)}}{{{y^2}}}dy} } } \right)\\
	& = \mathcal{O}\left( {\int_0^1 {\left( {\int_0^y {\varphi \left( {{{\log }_\lambda }x} \right)dx} } \right)\frac{{\varpi \left( y \right)}}{{{y^2}}}dy} } \right)\\
	& = \mathcal{O}\left( 1 \right).
\end{align*}
Then the conclusion follows from Theorem \ref{c1conjugacy} directly.
\end{proof}

Integrability  condition \eqref{jifenyouxian} on modulus of continuity seems to be complicated, but for many common cases one could simplify \eqref{jifenyouxian}  explicitly through asymptotic analysis, see the Corollary \ref{coro1} given below. As it can be seen that, to obtain $ C^1 $ conjugacy on $ \mathbb{T}^1 $, the larger $ k_n $ is, the stronger the required regularity has to be.

\begin{corollary}\label{coro1} The $ C_{2,\varpi} $ mapping $ T $ in Theorem \ref{varpitype} admits $ C^1 $  conjugacy, if the corresponding irrational rotation number $ \rho\in\mathbb{T}^1 $ is of the following type and the modulus of continuity satisfies the integrability condition:
\begin{itemize}
	\item[(C1)] The constant type, i.e., $ k_{n+1}=\mathcal{O}(1) $:
	\[\int_0^1 {\frac{{\varpi \left( y \right)}}{y}dy}  <  + \infty .\]
	For instance, the Lipschitz type $ \varpi\sim x $; the H\"older type $ \varpi\sim x^\alpha$ with any $  0<\alpha<1 $; the Logarithmic H\"older type $ \varpi  \sim {\left( {\log {x^{ - 1}}} \right)^{ - \alpha }} $ with any  $  \alpha>1 $; and even
	\[\varpi \sim \frac{1}{{(\log {x^{ - 1}})(\log \log {x^{ - 1}}) \cdots {{(\underbrace {\log  \cdots \log }_\ell {x^{ - 1}})}^{1 + \sigma }}}}\]
	for any $ \ell \in \mathbb{N}^+ $ and $ \sigma>0 $.
	
	\item[(C2)] The polynomial type, i.e.,  $ k_{n+1}=\mathcal{O}(n^{\nu}) $ with some $ \nu>0 $:
	\[\int_0^1 {\frac{{{{\left( { - \log y} \right)}^\nu }\varpi \left( y \right)}}{y}dy}  <  + \infty .\]
	For instance, the Lipschitz  type $ \varpi\sim x $; the H\"older type $ \varpi\sim x^\alpha$ with  any $  0<\alpha<1 $; the Logarithmic H\"older type $ \varpi  \sim {\left( {\log {x^{ - 1}}} \right)^{ - \alpha }} $ with  any $  \alpha>\nu+1 $; and even
	\[\varpi  \sim \frac{1}{{{{(\log {x^{ - 1}})}^{\nu  + 1}}(\log \log {x^{ - 1}}) \cdots {{(\underbrace {\log  \cdots \log }_\ell {x^{ - 1}})}^{1 + \sigma }}}}\]
	for any $ \ell \in \mathbb{N}^+ $ and $ \sigma>0 $.
	
	\item[(C3)] The exponential type, i.e.,    $ k_{n+1}=\mathcal{O}(a^n) $ with some $ 1<a<\lambda^{-1} $. Then for any fixed $ b>0 $ arbitrarily small, one  has to require that
	\[\int_0^1 {\frac{{\varpi \left( y \right)}}{{{y^{1 + b}}}}dy}  <  + \infty .\]
	For instance, the Lipschitz type $ \varpi\sim x $; the H\"older type $ \varpi\sim x^\alpha$ with  any $  0<\alpha<1 $.
\end{itemize}
\end{corollary}
\begin{remark}\label{rema416}
(C1) is the same as the  Dini condition \eqref{DINI}, see \eqref{deltapiao} and \eqref{powermoc}  for more examples. (C2)  shows that arbitrary H\"older continuity (or even Logarithmic H\"older continuity $ (\log(x^{-1}))^{-1-\varepsilon} $ with any $ \varepsilon>0 $) is sufficient for $ C^1 $ conjugacy of almost irrational rotations.  
\end{remark}
\begin{proof}
Explicit examples are easy to verify and we thus omit the proof here.
\begin{itemize}
	\item[(C1)] Note $ \varphi(x)=1 $. Therefore by \eqref{jifenyouxian} we get
	\begin{align*}
		\int_0^1 {\left( {\int_0^y {\varphi \left( {{{\log }_\lambda }x} \right)dx} } \right)\frac{{\varpi \left( y \right)}}{{{y^2}}}dy}  &= \int_0^1 {\left( {\int_0^y {1dx} } \right)\frac{{\varpi \left( y \right)}}{{{y^2}}}dy} \\
		& = \int_0^1 {\frac{{\varpi \left( y \right)}}{y}dy}  <  + \infty.
	\end{align*}	
	
	\item[(C2)] Note $ \varphi \left( x \right) = {x^\nu } $ with some $ \nu>0 $, and
	\begin{align*}
		\int_0^y {\varphi \left( {{{\log }_\lambda }x} \right)dx}  &= \int_0^y {{{\left( {{{\log }_\lambda }x} \right)}^\nu }dx} \\
		& = \mathcal{O}\left( {\int_0^y {{{\left( { - \log x} \right)}^\nu }dx} } \right)\\
		& = \mathcal{O}\left( {\int_{ - \log y}^{ + \infty } {{z^\nu }{e^{ - z}}dz} } \right)\\
		& = \mathcal{O}\left( {{{\left( { - \log y} \right)}^\nu }y} \right).
	\end{align*}
	Then by \eqref{jifenyouxian} we have
	\begin{align*}
		\int_0^1 {\left( {\int_0^y {\varphi \left( {{{\log }_\lambda }x} \right)dx} } \right)\frac{{\varpi \left( y \right)}}{{{y^2}}}dy}  &= \mathcal{O}\left( {\int_0^1 {\frac{{{{\left( { - \log y} \right)}^\nu }\varpi \left( y \right)}}{y}dy} } \right)\\
		& = \mathcal{O}\left( 1 \right).
	\end{align*}
	
	\item[(C3)] For any fixed $ b>0 $ arbitrarily small, let $ {a_ * } = {e^{ - b\ln \lambda }} \in \left( {1,{\lambda ^{ - 1}}} \right) $. Note that $ \varphi \left( x \right) = a_ * ^x $, and
	\[\int_0^y {\varphi \left( {{{\log }_\lambda }x} \right)dx}  = \int_0^y {{x^{ - b}}dx}  = \mathcal{O}\left( {{y^{1 - b}}} \right).\]
	Then it follows that
	\begin{align*}
		\int_0^1 {\left( {\int_0^y {\varphi \left( {{{\log }_\lambda }x} \right)dx} } \right)\frac{{\varpi \left( y \right)}}{{{y^2}}}dy}  &= \int_0^1 {\frac{{\varpi \left( y \right)}}{{{y^{1 + b}}}}dy} \\
		&= \mathcal{O}\left( 1 \right).
	\end{align*}
\end{itemize}
\end{proof}

Additionally,  we could further discuss the regularity of the conjugation $ \phi $ based on the above Theorem \ref{varpitype} as a  byproduct. We emphasize that the integrability condition \eqref{jifenyouxian} is indeed crucial. As pointed out in Remark \ref{666}, $ \phi $ has a modulus of continuity $ \varpi_h $, but  it's hard to give an explicit form in general. The following Theorem \ref{higherregu} provides a way to construct an explicit one $ \widetilde{\varpi} $ such that $ \widetilde{\varpi}  \gtrsim \varpi_h $. As shown by Corollary \ref{CORO11}, optimal estimates in special cases could be derived, but  it seems difficult to obtain optimality in general settings. We also hold the opinion that the cross-ratio distortion estimates can extend the optimal results in \cite{Khanin09Invent} based on Diophantine rotations, but we will not discuss that here.

\begin{theorem}[Higher regularity]\label{higherregu}
Define
\[\mathcal{R} \left( n \right): = \int_0^{{\lambda ^n}} {\left( {\varphi \left( {{{\log }_\lambda }x} \right)\int_x^1 {\frac{{\varpi \left( y \right)}}{{{y^2}}}dy} } \right)dx} .\]
If $ \widetilde \varpi:=\mathcal{R}  \circ {{ \Delta }^{ - 1}}\left(  \cdot  \right) $ is indeed a modulus of continuity, then $ \phi$ in Theorem \ref{varpitype} belongs to $   {C_{1,\widetilde \varpi }}( {{\mathbb{T}^1}} ) $.
\end{theorem}
\begin{remark}\label{daoshuji}
Obviously $ \mathcal{R} $ is well-defined through the integrability condition \eqref{jifenyouxian} and satisfies $ \mathcal{R}(0+)=0 $ by applying Fubini's Theorem and Cauchy's Theorem. It is easy to verify that $ \widetilde \varpi \left( {0 + } \right) = 0 $. Notice that $ \Delta $ can be replaced by a function $ \widetilde{\Delta} $ that satisfies $ \widetilde{\Delta}\leqslant \Delta $, e.g., $ \widetilde \Delta \left( n \right): = \prod\nolimits_{j =  - 1}^{n + 1} {{{\left( {{k_{j + 2}} + 1} \right)}^{ - 1}}}  $.
\end{remark}
\begin{proof}
In this case only the regularity of the density $ h $ needs to be estimated. Let $ {\eta _1},{\eta _2} \in {\mathbb{T}^1} $ be fixed such that they are sufficiently close to each other. Then there exist a unique $ n\in \mathbb{N} $ such that $ {\Delta _n} \leqslant \left| {\phi \left( {{\eta _1}} \right) - \phi \left( {{\eta _2}} \right)} \right| < {\Delta _{n + 1}} $, and let $ k \in \mathbb{N}^+ $ be the largest number such that $ k{\Delta _n} \leqslant \left| {\phi \left( {{\eta _1}} \right) - \phi \left( {{\eta _2}} \right)} \right| $. Obviously $ 1 \leqslant k \leqslant {k_{n + 1}} $. In view of Lemma \ref{lemmaifenjie}, Theorem \ref{denjoymodulus}, the continuity of $ h $ (see Theorem \ref{c1conjugacy}) and the integrability condition \eqref{jifenyouxian},  we derive the following similar to that in Theorem \ref{varpitype}:
\begin{align*}
	\left| {\log h\left( {{\eta _2}} \right) - \log h\left( {{\eta _1}} \right)} \right| &= \mathcal{O}\left( {k{\tau _n} + \sum\limits_{s = n + 1}^\infty  {{k_{s + 1}}{\tau _s}} } \right)\\
	& = \mathcal{O}\left( {\sum\limits_{s = n}^\infty  {{k_{s + 1}}{\tau _s}} } \right)\\
	& = \mathcal{O}\left( {\sum\limits_{s = n}^\infty  {\varphi \left( s \right) \cdot {\lambda ^s}\int_{{\lambda ^s}}^1 {{y^{ - 2}}\varpi \left( y \right)dy} } } \right)\\
	& = \mathcal{O}\left( {\int_0^{{\lambda ^n}} {\left( {\varphi \left( {{{\log }_\lambda }x} \right)\int_x^1 {{y^{ - 2}}\varpi \left( y \right)dy} } \right)dx} } \right)\\
	&= \mathcal{O}\left( {\mathcal{R} \left( n \right)} \right).
\end{align*}
The same estimate is true for $ \left| {h\left( {{\eta _2}} \right) - h\left( {{\eta _1}} \right)} \right| $ because $ \log h\left( {{\eta _1}} \right) = \mathcal{O}\left( 1 \right) $. Since $ \phi $ has already been proved to be a diffeomorphism in Theorem \ref{varpitype}, we therefore get
\begin{align*}
	\left| {h\left( {{\eta _2}} \right) - h\left( {{\eta _1}} \right)} \right| &= \mathcal{O}\left( {\mathcal{R} \left( n \right)} \right)\\
	&= \mathcal{O}\left( {\widetilde \varpi \circ { \Delta \left( n \right)} } \right)\\
	& = \mathcal{O}\left( {\widetilde \varpi \left( {k{\Delta _n}} \right)} \right)\\
	&= \mathcal{O}\left( {\widetilde \varpi \left( {\left| {\phi \left( {{\eta _1}} \right) - \phi \left( {{\eta _2}} \right)} \right|} \right)} \right)\\
	& = \mathcal{O}\left( {\widetilde \varpi \left( {\left| {{\eta _1} - {\eta _2}} \right|} \right)} \right),
\end{align*}
which gives that $ h \in {C_{0,\widetilde \varpi }}( {{{\mathbb{T}^1}}} ) $, and thus one arrives at $ \phi  \in {C_{1,\widetilde \varpi }}( {{{\mathbb{T}^1}}} ) $.

The proof is completed.
\end{proof}

It should be emphasized that we always study circle  diffeomorphisms in terms of partial quotients $ k_n $ throughout this paper. Since $ \Delta(n) $ is uniquely determined for the known $ k_n $, some precise estimates might be obtained by applying Theorem \ref{higherregu} at this point.  As an illustration, let us consider the simplest case $ k_n=\mathcal{O}(1) $ (the 1-type irrationality), i.e., there exist $ K_1, \ldots ,K_\nu\in \mathbb{N}^+ $ with $ \nu \in \mathbb{N}^+ $ such that $ k_n\in \mathcal{S}=\left\{ {K_1, \ldots ,K_\nu} \right\} $, and we might be able to accurately estimate the asymptotic behavior of $ \Delta $ based on the probability of the numbers appearing in $ \mathcal{S} $ (or only  derive the rough estimates of $ k_n $ without considering the probability). One therefore obtains a $ C_{1,\widetilde{\varpi}} $ conjugation $ \phi $ for the mapping $ T $, and the modulus of continuity $ \widetilde{\varpi} $ might be more accurate than any known result, or even optimal.  We emphasize that the asymptotic analysis is crucial in dealing with the regularity of $ \phi $. Based on the above,  Corollary \ref{CORO11} below discusses regularity of both H\"older type and Logarithmic H\"older type, and only the simplest form of probability is studied (nevertheless, we still further improve the known optimal regularity results  so far, see Remark \ref{rema418}). It can be seen later that the probability at this point does not affect the latter, but the former depends on it explicitly. This phenomenon also demonstrates indirectly that it is necessary to study the classification. The more general irrationality such as  $ \varphi $-type cases could also be considered, but we will not discuss here.

\begin{corollary}[The 1-type irrationality]\label{CORO11}
Assume that  $ k_n\in \mathcal{S}=\left\{ {K_1, \ldots ,K_\nu} \right\} $ ($ \nu \geqslant 2 $, and $ {K_i} \ne {K_j} $ for $ i \ne j $) satisfy the nonuniform distribution:
\begin{equation}\label{nonuniform}
	\mathop {\lim }\limits_{N \to  + \infty } \frac{{\left| {\left\{ {n = {K_j}:1 \leqslant n \leqslant N} \right\}} \right|}}{N} = {p_j} \in \left( {0,1} \right),\quad1 \leqslant j \leqslant \nu,
\end{equation}
where $ \sum\nolimits_{j = 1}^\nu  {{p_j}}  = 1 $. Then there hold:
\begin{itemize}
	\item[(C4)] If $ T\in C_{2,\varpi_1} ({\mathbb{T}^1})$ with $ {\varpi_1} \sim {x^\alpha }$ and $ 0<\alpha<1 $, then $ \phi \in C_{1,\widetilde{\varpi}_1} (\mathbb{T}^1) $ with $ {{\widetilde \varpi }_1} \sim {x^\beta } $, where
	\[\alpha\leqslant\beta : = \min \left\{ {1,\frac{{\alpha \log \lambda }}{{\log \vartheta }}} \right\},\quad\vartheta : = \prod\limits_{j = 1}^\nu  {K_j^{ - {p_j}}} .\]
	
	\item[(C5)] If $ T\in C_{2,\varpi_2} ({\mathbb{T}^1})$ with $ {\varpi _2} \sim {\left( { - \log x} \right)^{ - 1 - \sigma }}$ and $ \sigma>0 $, then $ \phi \in C_{1,\widetilde{\varpi}_2} (\mathbb{T}^1) $ with $ {{\widetilde \varpi }_2}\sim {\left( {\log {x^{ - 1}}} \right)^{ - \sigma }} $.
\end{itemize}

\end{corollary}
\begin{remark}\label{rema418}
Distribution \eqref{nonuniform} can actually be extended further to the strongly nonuniform case similar to that in Remark \ref{remark999}, where the probability depends on the position of the points, but we do not pursue that.
Once  \eqref{nonuniform} is removed, the diffeomorphism $ \phi $ in (C4) can be of $ C_{1,\varpi^*}({\mathbb{T}^1}) $ with $ \gamma  = \min \left\{ {1,\frac{{ - \alpha \log \lambda }}{{\log M}}} \right\} $, where $ M = \mathop {\max }\nolimits_{1 \leqslant j \leqslant \nu } {K_j } $ (since positive integers are uniformly distributed). Similar to the following proof we have $ \mathop {\inf }\nolimits_{M \in {\mathbb{N}^ + }} \gamma  \geqslant \alpha  $. This implies that Diophantine's type of exponent $ 2 $ ($ k_n=\mathcal{O}(1) $) admits $ C^{1,\alpha} $ conjugacy, which has been shown to be optimal in \cite{ETDScon}, see also \cite{Khanin09Invent}, and it shows that our theorem is a little bit more accurate than the known optimal results. Additionally,  (C5) is independent of the probability.
\end{remark}
\begin{proof}
We first establish an explicit expression about $ \mathcal{R} $, based on a universal modulus of continuity $ \varpi $ which is strictly weaker than the Lipschitz type (because it is trivial and can be discussed separately, we therefore assume that  $ x/\varpi \left( x \right) $ is monotonically decreasing to $ 0 $ as $ x \to 0^+ $ without loss of generality), and $ k_{n}=\mathcal{O}(1) $ (without considering the probability).

Note that $ \mathop {\inf }\nolimits_{0 < z < {2^{ - 1}}} \int_z^1 {d\left( {{y^{ - 1}}\varpi \left( y \right)} \right)}  < 0 $ by the assumption we make previously, and
\[\int_z^1 {{y^{ - 2}}\varpi \left( y \right)dy}  = {z^{ - 1}}\varpi \left( z \right) - \int_z^1 {d\left( {{y^{ - 1}}\varpi \left( y \right)} \right)} .\]
Then it follows from L'Hospital's rule that
\[\mathop {\lim }\limits_{z \to {0^ + }} \frac{{\int_0^z {{y^{ - 1}}\varpi \left( y \right)dy} }}{{z\int_z^1 {{y^{ - 2}}\varpi \left( y \right)dy} }} = \mathop {\lim }\limits_{z \to {0^ + }} \frac{{{z^{ - 1}}\varpi \left( z \right)}}{{\int_z^1 {{y^{ - 2}}\varpi \left( y \right)dy}  - {z^{ - 1}}\varpi \left( z \right)}} =  + \infty .\]
This gives
\begin{align*}
	\mathcal{R} \left( n \right) &= \int_0^{{\lambda ^n}} {\left( {\varphi \left( {{{\log }_\lambda }x} \right)\int_x^1 {{y^{ - 2}}\varpi \left( y \right)dy} } \right)dx} \\
	& = \int_0^{{\lambda ^n}} {\left( {\int_x^1 {{y^{ - 2}}\varpi \left( y \right)dy} } \right)dx}  \\
	& = \int_0^{{\lambda ^n}} {\left( {\int_0^y {{y^{ - 2}}\varpi \left( y \right)dx} } \right)dy}  + \int_{{\lambda ^n}}^1 {\left( {\int_0^{{\lambda ^n}} {{y^{ - 2}}\varpi \left( y \right)dx} } \right)dy}  \\
	& = \int_0^{{\lambda ^n}} {{y^{ - 1}}\varpi \left( y \right)dy}  + {\lambda ^n}\int_{{\lambda ^n}}^1 {{y^{ - 2}}\varpi \left( y \right)dy}  \\
	& = \mathcal{O}\left( {\int_0^{{\lambda ^n}} {{y^{ - 1}}\varpi \left( y \right)dy} } \right).
\end{align*}

Next we provide an estimate of $ \Delta $ based on \eqref{nonuniform} and Remark \ref{daoshuji}. Let $ c>0 $ be a generic constant that does not affect the estimates. Since $ {\Delta _n} = {k_{n + 2}}{\Delta _{n + 1}} + {\Delta _{n + 2}} < \left( {{k_{n + 2}} + 1} \right){\Delta _{n + 1}} $, we therefore obtain that $ {\Delta _n} \geqslant c{\vartheta ^n} $ (obviously $ 0<\vartheta <1 $), which gives
\[{\Delta ^{ - 1}}\left( x \right): = \frac{{\log {c^{ - 1}}x}}{{\log \vartheta }} \sim \frac{{\log {x^{ - 1}}}}{{\log {\vartheta ^{ - 1}}}}.\]
Note that $ {\Delta _n} = \mathcal{O}\left( {{\lambda ^m}} \right) $ since $ \frac{{{\Delta _n}}}{{{\Delta _{n - m}}}} = \mathcal{O}\left( {{\lambda ^m}} \right) $.	This implies that $ \lambda \geqslant \vartheta $ (in fact $ \lambda >\frac{1}{2} \geqslant \vartheta $ by Statement \ref{item:b}), and therefore $ \beta  := \min \left\{ {1,\frac{{\alpha \log \lambda }}{{\log \vartheta }}} \right\} \geqslant \alpha  $.

Finally, denote $ \mathcal{R}_1 $ and $ \mathcal{R}_2 $ with respect to $ \varpi_1 $ and $ \varpi_2 $, respectively.  For (C4), direct calculation gives that
\[{\mathcal{R}_1}\left( n \right) = \mathcal{O}\left( {\int_0^{{\lambda ^n}} {{y^{ - 1 - \alpha }}dy} } \right){\text{ = }}\mathcal{O}\left( {{\lambda ^{\alpha n}}} \right),\]
therefore,
\[{{\widetilde \varpi }_1} \sim {\mathcal{R}_1} \circ {\Delta ^{ - 1}} = \mathcal{O}\left( {{\lambda ^{\alpha  \cdot \frac{{\log {x^{ - 1}}}}{{\log {\vartheta ^{ - 1}}}}}}} \right) = \mathcal{O}\left( {{x^\beta }} \right).\]
As to (C5), one notices that
\[{\mathcal{R}_2}\left( n \right) = \mathcal{O}\left( {\int_0^{{\lambda ^n}} {{y^{ - 1}}{{\left( { - \log y} \right)}^{ - 1 - \sigma }}dy} } \right) = \mathcal{O}\left( {{n^{ - \sigma }}} \right)\]
by asymptotic analysis. This leads to
\[{{\widetilde \varpi }_2} \sim {\mathcal{R}_2} \circ {\Delta ^{ - 1}} = \mathcal{O}\left( {{{\left( {\log {x^{ - 1}}} \right)}^{ - \sigma }}} \right).\]
One can easily verify that both $ {\widetilde \varpi }_1 $ and $ {\widetilde \varpi }_2 $ are modulus of continuity, then the conclusions follow directly from Theorem \ref{higherregu}.
\end{proof}

We conclude this section with potential applications of our results. Very recently, the first and third authors investigated the exponential convergence of Birkhoff averages under certain special weights for quasi-periodic and almost periodic systems (as well as more general cases), which breaks the inherently slow convergence rate typical of classical ergodic theory; see, for example, \cite{TL24a,TL25b,TL26a}. While those works focused on reduced simple systems, concrete problems necessitate a smooth conjugation theorem as a mediator---such as the classical KAM theorem. Theorems \ref{varpitype} and \ref{higherregu} of the present paper can also serve this purpose, which will be the subject of our future research.

\subsection{Optimality of our integrability condition}\label{Optimality about our integrability condition}
In this section, we will show certain optimality of our integrability condition \eqref{jifenyouxian}, namely
\[	\int_0^1 {\left( {\int_0^y {\varphi \left( {{{\log }_\lambda }x} \right)dx} } \right)\frac{{\varpi \left( y \right)}}{{{y^2}}}dy}  <  + \infty .\]
Let us start by reviewing some classic and important results.

\begin{itemize}
\item[(D1)]  Yoccoz's result \cite[Theorem 1.1, p. 126]{MR1924912}   concludes that, if the mapping $ T $ considered throughout this paper is only $ C^2 $, and $ \rho \in \mathbb{T}^1 $ is a number of the constant type (i.e., the $ 1 $-type in our terminology, or  equivalently, $ k_n =\mathcal{O}(1)$), then the conjugation $ \phi $ to the rotation with $ \rho $ is indeed  absolutely continuous (therefore, must be differentiable a.e.). This is somewhat different from the classical Denjoy theory in Section \ref{Subdenjoy} since the regularity for both $ T $ and $ \phi $ are higher. \textit{However, $ C^1 $ conjugacy has not yet been achieved at this point.} 

\item[(D2)] Meanwhile,  Katznelson and Ornstein \cite{MR1036903}  emphasized that, the condition of boundedness of the continued-fraction coefficients of $ \rho $ is essential to obtain absolutely continuous conjugacy (stronger than topological conjugacy but weaker than $ C^1 $ conjugacy). Otherwise, for a given $ \rho \in \mathbb{T}^1 $ with unbounded coefficients, that is, there exists a sequence $ \{k_{n_j}\}_j $ satisfying $ \mathop {\lim }\nolimits_{j \to  + \infty } {k_{{n_j}}} =  + \infty  $, one can construct a $ C^2 $ mapping $ T $ on $ \mathbb{T}^1 $ such that the conjugation $ \phi $ to the rotation with $ \rho $ is \textit{purely singular}, i.e., maps a map of zero measure onto a map of full measure, let alone admitting continuity. Additionally, the mapping $ T $ can be chosen arbitrarily close to the rotation with $ \rho $. See related work, Hawkins and Schmidt \cite{MR662606},  Katznelson \cite{MR0581808} and Lazutkin \cite{MR0482815}.  Even for the rotation number of the constant type (the Diophantine index is $ 0 $), in order to preserve $ C^1 $ conjugacy, the regularity requirement of mapping $ T $ cannot be lower than that of $ C^2 $. More precisely, the counterexample in \cite[Appendix 3]{ETDScon}   shows that the regularity of $ C^{2-\varepsilon } $ with any $ 0< \varepsilon<1 $ for $ T $ cannot admit $ C^1 $ conjugacy at this point.

\item[(D3)] Katznelson and  Ornstein \cite[Section 3]{ETDScon}  also discussed circle diffeomorphisms of $ C^2 $'s type, provided certain modulus of continuity. There the  Denjoy-type inequality is not optimal. Besides, they were more concerned with the case based on Diophantine irrationality, and the higher H\"older regularity of conjugations.  Khanin and Teplinsky \cite{Khanin09Invent} established the optimal Denjoy-type inequality under $ C^{2+\alpha} $ smoothness, where $ 0< \alpha <1 $. However, as we mentioned in the introduction and Section \ref{SecModulus}, H\"older regularity is not sufficient to characterize continuity in general.  
\end{itemize}

Here we touch the optimality, which  can be summarized as:

\begin{itemize}
\item [(OP1)] Note that $ \mathcal{R} $ in Theorem \ref{higherregu} (regularity higher than $ C^1 $) is strongly related to the integrability condition \eqref{jifenyouxian}, and we derive more accurate estimates than the known optimal  regularity results (at least for the constant type, i.e., the $ 1 $-type irrationality), see  Corollary \ref{CORO11} and Remark \ref{rema418}. This shows that our integrability condition has certain optimality in the sense of obtaining higher regularity of $ C^1 $ conjugacy.

\item [(OP2)] Assuming \textit{only $ C^2 $ (without any extra H\"older continuity)} and for the constant type irrational numbers, \textit{we could indeed improve the absolute  continuous conjugacy in (D1) to $ C^1 $ conjugacy,} as long as the modulus of continuity   naturally possessed by $ D^2 T $ (Remark \ref{666}) satisfies the Dini condition \eqref{DINI}, see conclusion (C1) in Corollary \ref{coro1} and Remark \ref{rema416}.

\item [(OP3)] \textit{The counterexamples in (D2) also  illustrate the optimality of our integrability condition, that is, to preserve the  differentiable conjugacy, $ C^2 $ regularity for the mapping $ T $ cannot be further weakened.} And in particular, for $ \rho $ with  unbounded $ k_n $, we can still obtain $ C^1 $ conjugacy due to Theorem \ref{varpitype}, as long as the  irrationality and regularity satisfy the equilibrium in the integrability condition \eqref{jifenyouxian} (and thus this is not inconsistent with the counterexample constructed by Hawkins and Schmidt \cite{MR662606} in (D2)), while the mapping $ T $ we consider is still only $ C^2 $ (without any extra H\"older continuity). Besides, the counterexample  constructed by  Katznelson and  Ornstein \cite{ETDScon} in (D2) also reflects our optimality (note that the  constant type rotation corresponds to $ \varphi(x)=1 $ in our  integrability condition \eqref{jifenyouxian}). The above arguments show that our integrability condition is somewhat strong, \textit{in the sense of preserving $ C^1 $ conjugacy} (note the huge gap between pure singularity and differentiability).

\item[(OP4)] Recall (D3). Fortunately, via the cross-ratio distortion estimates introduced by Khanin and Teplinsky in \cite{Khanin09Invent}, we extend Denjoy-type inequality to the weakest case (only $ C^2 $ regularity) and obtain the optimal result (see \cite{Khanin09Invent}). Then we further apply this powerful tool to investigate $ C^1 $ conjugacy, and propose the optimal  integrability condition \eqref{jifenyouxian}, which reveals the \textit{explicit} relation between irrationality and regularity in preserving $ C^1 $ conjugacy \textit{for the first time}, and it is also obviously easier to verify, see Corollary \ref{coro1} for explicit examples. Additionally, due to the optimality of Denjoy-type inequality, the weaker but \textit{implicit} condition \eqref{414} in Remark \ref{remark999} is indeed better than the results in \cite{ETDScon}, in the sense of only  $ C^2 $  without any extra H\"older continuity for the mapping $ T $. It should be emphasized that our analysis process is much simpler than traditional approaches.
\end{itemize}

\section{Appendix}
\begin{lemma}\label{Knlemma}
Let $ \rho  = \left[ {{k_1}\left( \rho  \right),{k_2}\left( \rho  \right), \ldots } \right] \in {\mathbb{T}^1} \backslash \left\{ 0 \right\} $ be arbitrarily given. Assume that $ K_n \geqslant 0 $. Then
\begin{equation}\label{Lemma1.1}
	\sum\limits_{n = 1}^\infty  {{k_n}\left( \rho  \right){K_n}}  <  + \infty ,\quad\rho  - \text{a.e. in $ \mathbb{T}^1 $}
\end{equation}
holds if
\begin{equation}\label{Lemma1.2}
	\sum\limits_{n = 1}^\infty  {{K_n}\log n}  <  + \infty .
\end{equation}
Further, if $ \left\{ {{K_n}} \right\} $ is monotonically increasing, then \eqref{Lemma1.1} is in fact equivalent to \eqref{Lemma1.2}.
\end{lemma}
\begin{proof}
See details from Lemma A.2.1 and its Remark  in \cite{ETDScon}.
\end{proof}

\section*{Acknowledgements} 
The authors would like to sincerely thank the anonymous referees for their valuable suggestions and comments, which significantly improved the paper. Z. Tong  was supported by the China Postdoctoral Science Foundation (Grant No. 2025M783102). S. Xiao was supported by the Fundamental Research Funds for the Central Universities (Grant No. 044420250103). Y. Li was supported in part by the National Natural Science Foundation of China (Grant Nos. 12071175, 12471183 and 12531009).

\end{document}